\documentclass{amsart}

\usepackage{amssymb}
\usepackage{hyperref}
\usepackage{graphicx}
\usepackage{version}	
\usepackage{tikz}

\title{Accessible hyperbolic components in anti-holomorphic dynamics}

\author{Hiroyuki Inou \and Tomoki Kawahira}

\newtheorem{thm}{Theorem}[section]
\newtheorem{lem}[thm]{Lemma}

\newtheorem{cor}[thm]{Corollary}

\theoremstyle{definition}
\newtheorem{defn}[thm]{Definition}
\newtheorem*{ackn}{Acknowledgement}

\theoremstyle{remark}
\newtheorem{rem}[thm]{Remark}

\DeclareMathOperator{\Int}{Int}

\newcommand{\C}{\mathbb{C}}
\newcommand{\D}{\mathbb{D}}
\newcommand{\R}{\mathbb{R}}
\newcommand{\Z}{\mathbb{Z}}

\newcommand{\cH}{\mathcal{H}}
\newcommand{\cM}{\mathcal{M}}
\newcommand{\cR}{\mathcal{R}}
\newcommand{\cU}{\mathcal{U}}

\DeclareMathOperator{\im}{Im}

\DeclareMathOperator{\Koenigs}{Koe}
\DeclareMathOperator{\KoenigsInv}{Poi}
\DeclareMathOperator{\Boettcher}{B\ddot ot}
\DeclareMathOperator{\Fatou}{Fat}

\newcommand{\FatouInv}{\Psi} 

\DeclareMathOperator{\rep}{rep}
\DeclareMathOperator{\attr}{attr}

\DeclareMathOperator{\diam}{diam}
\DeclareMathOperator{\area}{Area}

\begin{document}

\maketitle

\begin{abstract}
  The tricorn, the connectedness locus of the anti-holomorphic quadratic family, is known to be non-locally connected.
  The boundary of every hyperbolic component of odd period contains arcs that are inaccessible from the complement of the tricorn. As the period increases, the decorations become more and more complicated, and it seems natural to think that every hyperbolic component of sufficiently large and odd period is inaccessible.
  Contrary to this expectation, we show that the tricorn contains infinitely many hyperbolic components that are accessible from the complement.
\end{abstract}

\section{Introduction}

The Mandelbrot set $\cM$ is the connectedness locus of the (holomorphic) quadratic family $Q_c(z)=z^2 + c$, i.e., the set of parameters $c$ for which the filled Julia set
\[
  K(Q_c) = \{z \in \C \mid \{Q_c^n(z)\}_{n \ge 0}\text{: bounded}\}
\]
is connected.
For the anti-holomorphic quadratic family $f_c(z)=\bar{z}^2+c$, the connectedness locus $\cM^*=\{c \in \C \mid K(f_c)\text{: connected}\}$ is called the \emph{tricorn} \cite{MR1181083} or the \emph{Mandelbar set} \cite{MR1020441}.

Milnor \cite{MR1181083} found tricorn-like sets in the parameter space of real cubic polynomials, and presented the tricorn as a prototype of them.
In fact, tricorn-like sets naturally appear in the parameter loci of families of (holomorphic) rational maps having symmetry by an anti-holomorphic involution.
For example, Bonifant, Buff and Milnor \cite{MR3772619} studied cubic antipode-preserving maps and found some tricorn-like sets, and Lodge and Mukherjee \cite{MR4211202} studied real degree four Newton maps and showed the existence of invisible tricorn-like sets.
The dynamics of some anti-holomorphic correspondences, called Schwarz reflections, have been intensively studied recently and some tricorn-like structures are discussed (see \cite{MR4706575}, \cite{Lee:2018aa} and \cite{MR4250610}).

Since the second iterate of anti-holomorphic dynamics is holomorphic,
every anti-holomorphic dynamics can be regarded as a subfamily of some family of holomorphic dynamics.
However, it is not a holomorphic subfamily; the second iterate of an anti-holomorphic family depends just real analytically on parameters. For example, $f_c^2(z) = (z^2+\bar{c})^2+c$ contains both $c$ and $\bar{c}$, so it is a real two-parameter family of holomorphic quartic polynomials.
Hence such a one-complex-parameter family of anti-holomorphic dynamics can inherit many wild structures in a holomorphic two-dimensional family, which do not appear in holomorphic one-dimensional families.
From this point of view, such a family can be considered as a ``1.5-dimensional family'' and it acts as a stepping stone for understanding higher-dimensional families.

It is well-known that $\cM$ is connected and full \cite{MR762431},
and the famous \emph{MLC conjecture} asserts that $\cM$ is locally connected.
If the MLC conjecture holds, then by Carath\'eodory 's theorem, $\partial \cM$ is a quotient of the circle and hence every point in the boundary is the landing point of some parameter ray.
In particular, every point in $\partial \cM$ is \emph{accessible} from the complement of $\cM$; that is, there exists a path in $\C \setminus \cM$ which converges to that point.

On the contrary, it is known that $\cM^*$ is not locally connected, although it is connected and full \cite{MR1235477}.
Indeed, there are real-analytic arcs of positive length in $\partial \cM^*$ which consist of inaccessible points from the complement \cite{hubbard_multicorns_2014}.
Such arcs are contained in the boundary of \emph{hyperbolic components} of odd period. Moreover, it is also known that every parameter ray for $\cM^*$ accumulating to a hyperbolic component of odd period at least three does not converge to a point \cite{MR3502067}.

Decorations of $\cM^*$ accumulating to a hyperbolic component of odd period become more and more complicated as the period increases, so it seems natural to think that such a hyperbolic component has no accessible boundary point if the period is sufficiently large (see the decorations in Figure~\ref{fig:inaccessible} and Figure~\ref{fig:umbilical cord}).
Hubbard and Schleicher discussed the existence of an ``\emph{open beach}'' (a boundary arc of $\mathcal{M}^*$ with no decoration. See \S2 for more detail) in \cite{hubbard_multicorns_2014} and remarked as follows:
\textit{We do not know whether there are infinitely many parabolic arcs for which this occurs}.

In fact, the accessibility of a boundary point $c$ is related to the condition that the basin of infinity, when projected into the repelling Ecalle cylinder of the periodic point for $f_c$, contains a horizontal circle (we give a sufficient condition in Theorem~\ref{thm:accessibility}).
However, since the projection of the basin of infinity into the Ecalle cylinder is an annulus, whose modulus is reciprocal to the period of the hyperbolic component, it is unlikely to contain a horizontal circle when the period is large.

In this article, we prove that this speculation is in fact incorrect:
\begin{thm}
  \label{thm:main}
  There exist infinitely many hyperbolic components of odd periods in $\Int\cM^*$ which are accessible from $\C \setminus \cM^*$.
\end{thm}

Since the second iterate $f_c^2(z) = (z^2+\bar{c})^2+c$ is holomorphic in $z$,
the situation depends on the parity of the period.
For example, if a periodic point is of odd period, then its multiplier (as a periodic point for $f_c^2$) is always non-negative real.
Thus the boundary of a hyperbolic component of odd period consists of maps with a parabolic periodic point.
In fact, the boundary arcs consist of three \emph{cusps} and three arcs connecting them, which are called \emph{parabolic arcs}
(see \cite{MR1020441}, \cite{MR2020986} and \cite{Mukherjee:2014ab} for more details).
When the period is even, the situation looks similar to the case of the Mandelbrot set, except possibly the root is blown up to an arc with a cusp. Such a blown-up arc is a common boundary of two hyperbolic components, one of which has twice the period of the other.

The existence of parabolic arcs for odd period hyperbolic components is the source of complicated topological structure such as non-path connectivity \cite{hubbard_multicorns_2014}, \cite{MR4302166}, non-landing parameter rays \cite{MR3502067} and failure of self-similarity \cite{MR4302166} (see also \cite{10.1093/imrn/rnaa365}).

For the Mandelbrot set, the set of parabolic parameters of a given period is discrete.
Hence if a sequence of holomorphic quadratic polynomials has the property that the multiplier of the periodic point of a given period converges to one, then the sequence converges.
However, since a parabolic arc for the tricorn consists of parameters having a parabolic periodic point of the same period with multiplier one, such a sequence does not always converge for the anti-holomorphic quadratic family.

More precisely, a parabolic arc is a quasiconformal conjugacy class, and it is parametrized by the \emph{critical Ecalle height}.
Hence, when a sequence or an arc converges to a parabolic arc from outside of the hyperbolic component, the accumulation set corresponds to that of the critical Ecalle heights. In particular, if the critical Ecalle height does not converge,
it does not converge to a point. Such a non-convergence yields ``wiggly'' features of parameter rays, the ``umbilical cords'', and so on.

In fact, non-path connectivity and failure of self-similarity are proved by the existence of ``wiggly umbilical cords''; that is, the continuum connecting the origin to a given hyperbolic component of odd period (``the umbilical cord'') does not converge to a point. The only landing umbilical cords are those on the real line up to symmetry, and this shows that the natural dynamical correspondence (the \emph{straightening map}) between the unions of hyperbolic components and parts of the decorations is not continuous.

Similar to the fact that every parabolic parameter in the Mandelbrot set has exactly one or two parameter rays landing at it, for every parabolic arc in the tricorn, only one or two external angles accumulate it.
Here readers should notice that external rays are defined in terms of the natural real-analytic diffeomorphism between $\C \setminus \cM^*$ and $\C \setminus \overline{\D}$, which is not holomorphic \cite{MR1235477},
thus Carath\`eodory's theorem does not hold.

Accessibility and inaccessibility of hyperbolic components also shows that decorations of hyperbolic components of odd periods are topologically different in general (see Figure~\ref{fig:accessible} and Figure~\ref{fig:inaccessible}).
\begin{figure}[th]
  \centering
  \includegraphics[width=5cm]{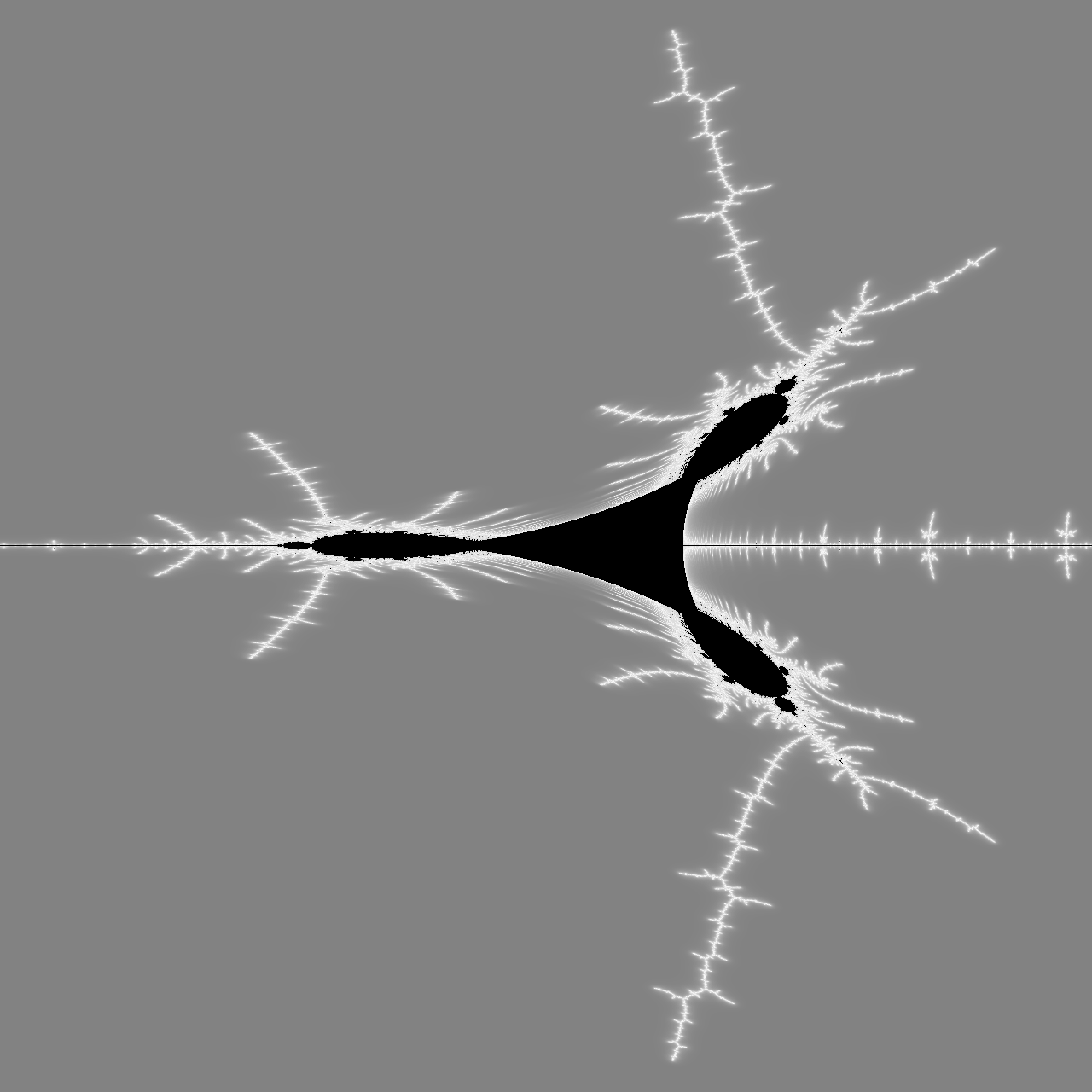}

  \smallskip

  \includegraphics[width=10cm]{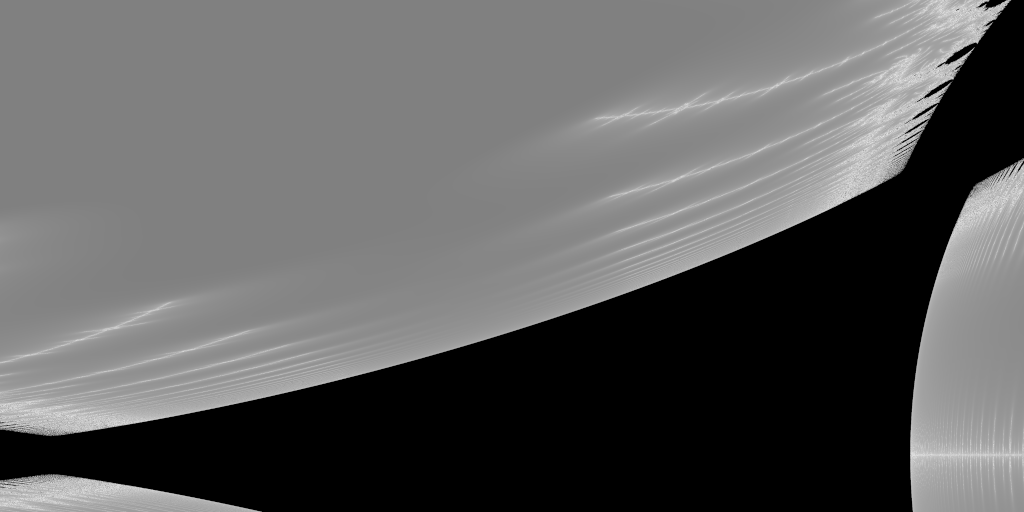}

  \caption{An accessible hyperbolic component of period 3 (centered at the airplane).}
  \label{fig:accessible}
\end{figure}

\begin{figure}[th]
  \centering
  \includegraphics[width=5cm]{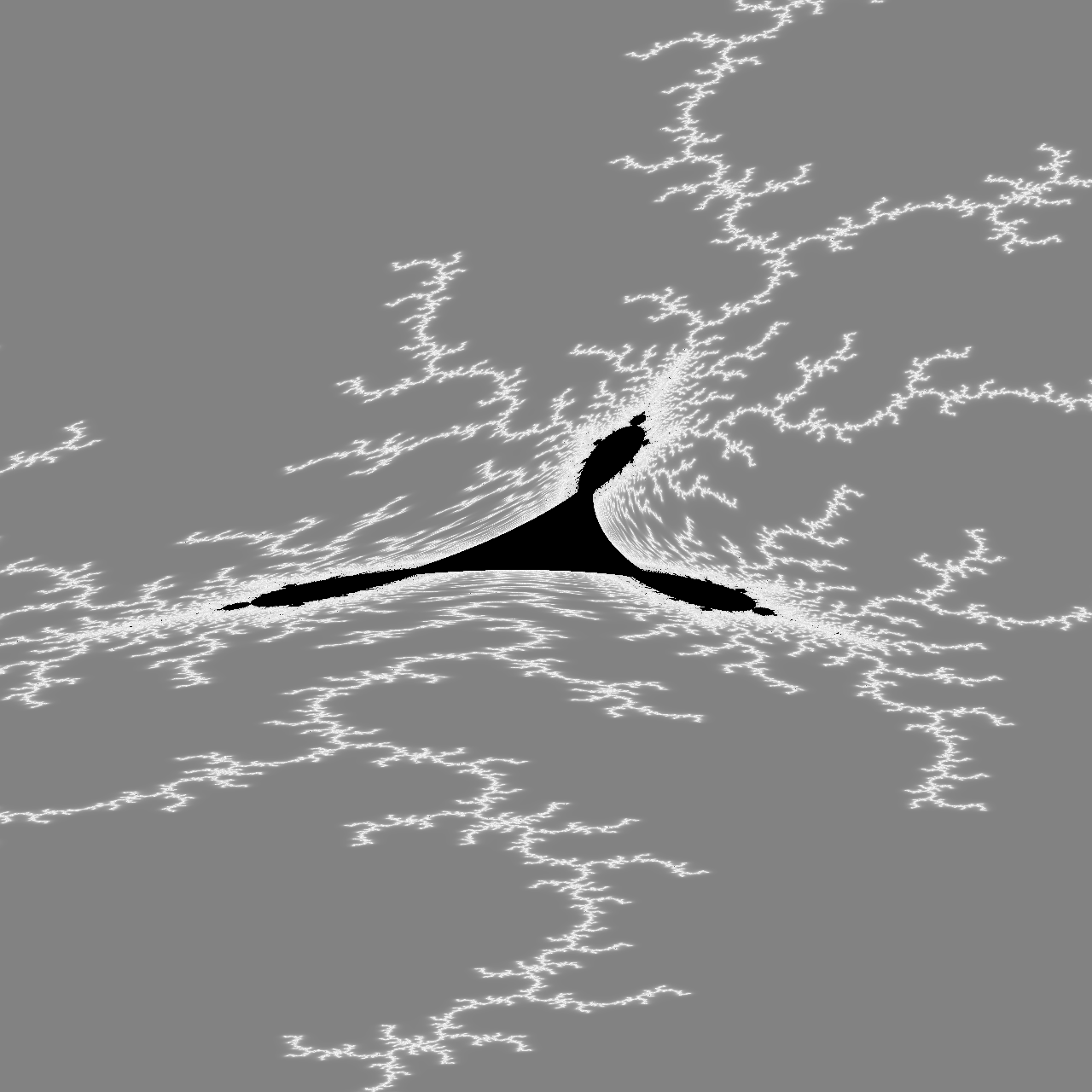}
  \includegraphics[width=5cm]{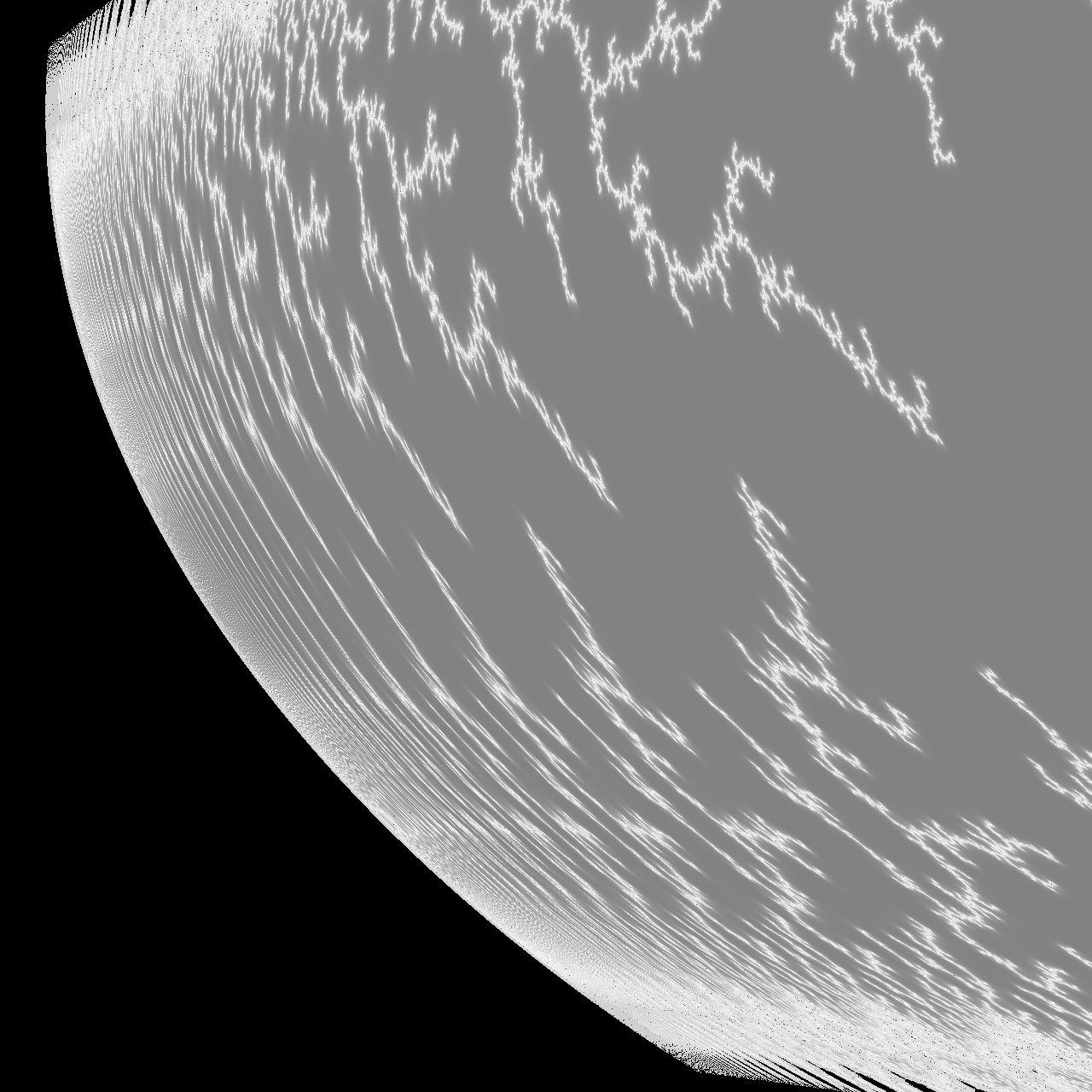}
  \includegraphics[width=5cm]{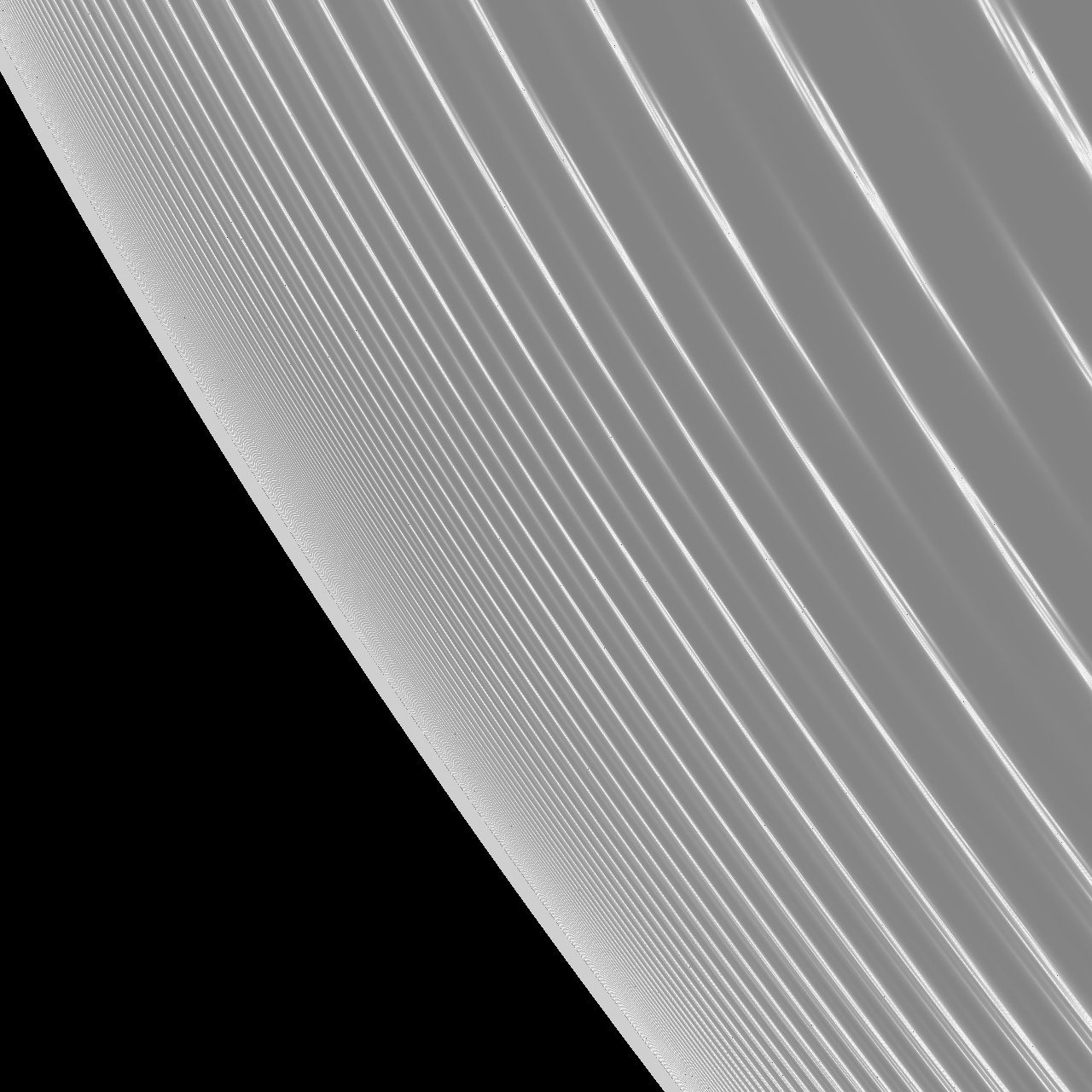}

  \caption{A hyperbolic component of period 9 which seems inaccessible.}
  \label{fig:inaccessible}
\end{figure}

The accessible hyperbolic components in Theorem~\ref{thm:main} are centered at real parameters and accumulates to the anti-Chebyshev map $f_{-2}$.
We can numerically see the decorations of the ``baby tricorn-like set'' centered at such a hyperbolic component converge to truncated dyadic parameter rays of the tricorn.
Also in the phase space for a parameter in the baby tricorn-like set,
the decorations attached to the small Julia set are also converge to truncated dyadic dynamical rays (see Figure~\ref{fig:decorations}).
\begin{figure}[th]
  \centering
  \usetikzlibrary{arrows.meta}
  \begin{tikzpicture}[every node/.style={inner sep=0,outer sep=0},
    pre/.style={<-,shorten <=2pt,>={Stealth[round]},semithick}]
    \node at (-0.75/4*5,0) {\includegraphics[width=5cm]{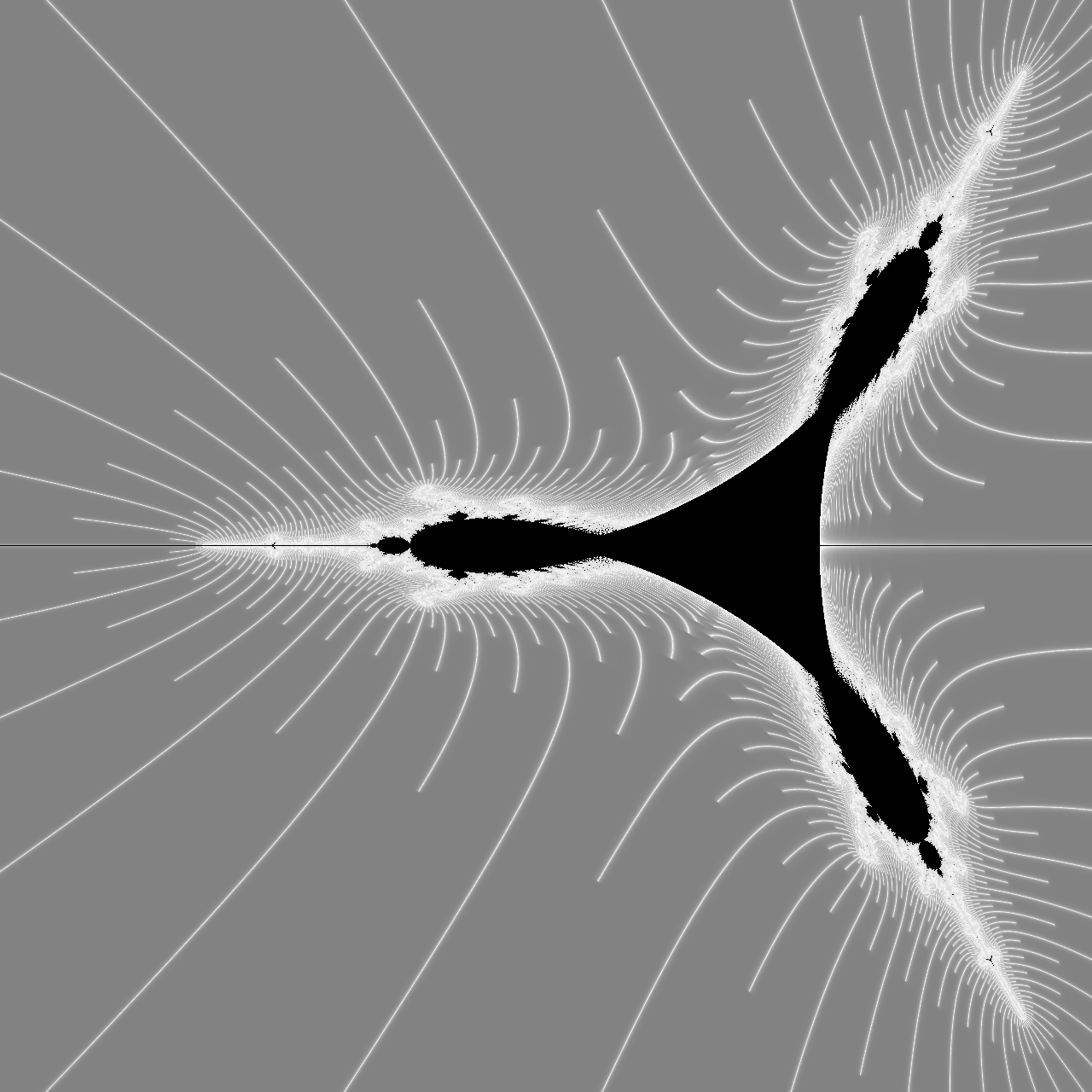}};
    \coordinate (O) at (0,0);
    \coordinate (A) at (-0.25*0.5/4*5,{0.25*sqrt(3)/2/4*5})
    edge [pre,white] (O);
    \fill [red] (A) circle (2pt);
  \end{tikzpicture}
  \includegraphics[width=5cm]{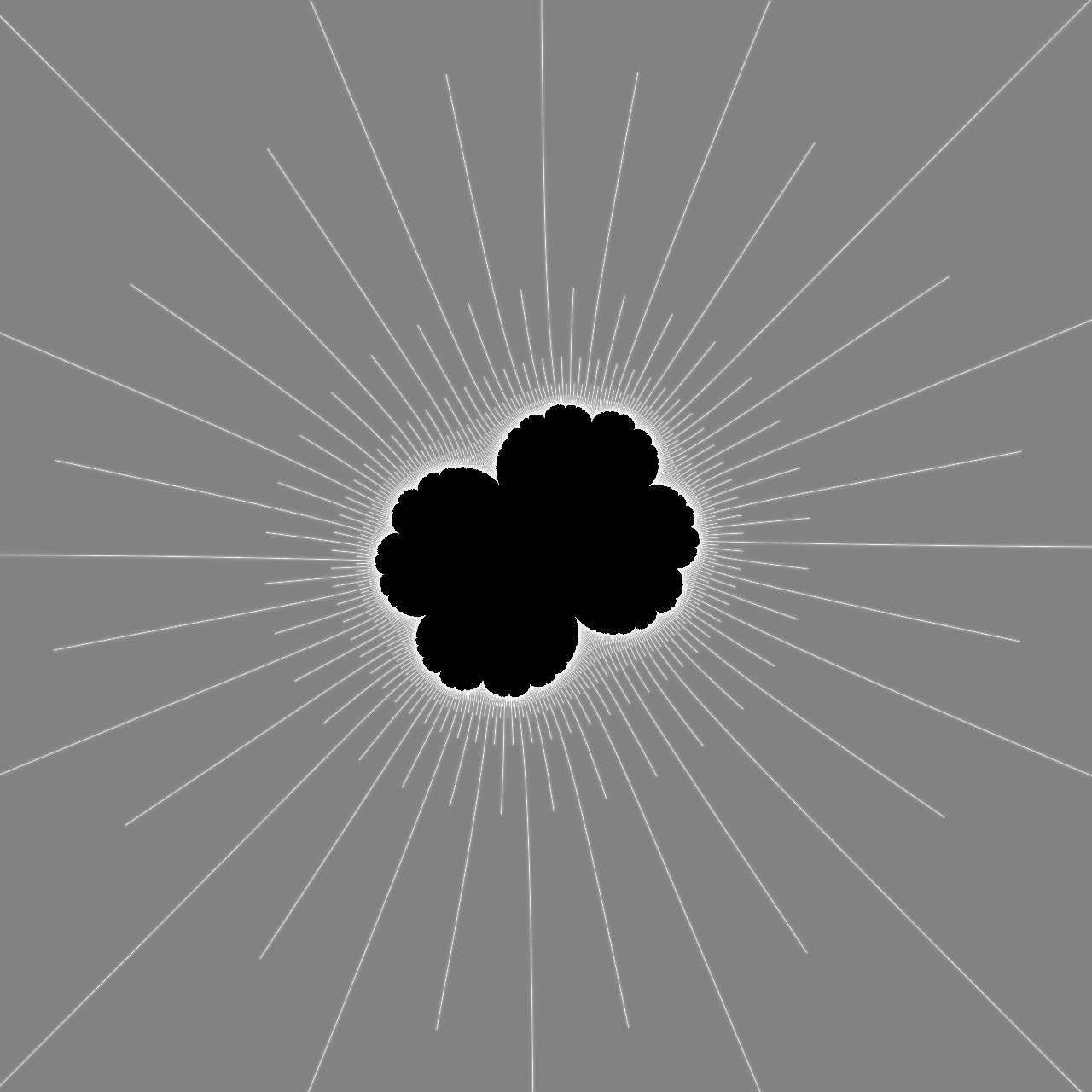}

  \usetikzlibrary{arrows.meta}
  \begin{tikzpicture}[every node/.style={inner sep=0,outer sep=0},
    pre/.style={<-,shorten <=2pt,>={Stealth[round]},semithick}]
    \node at (-0.75/4*5,0) {\includegraphics[width=5cm]{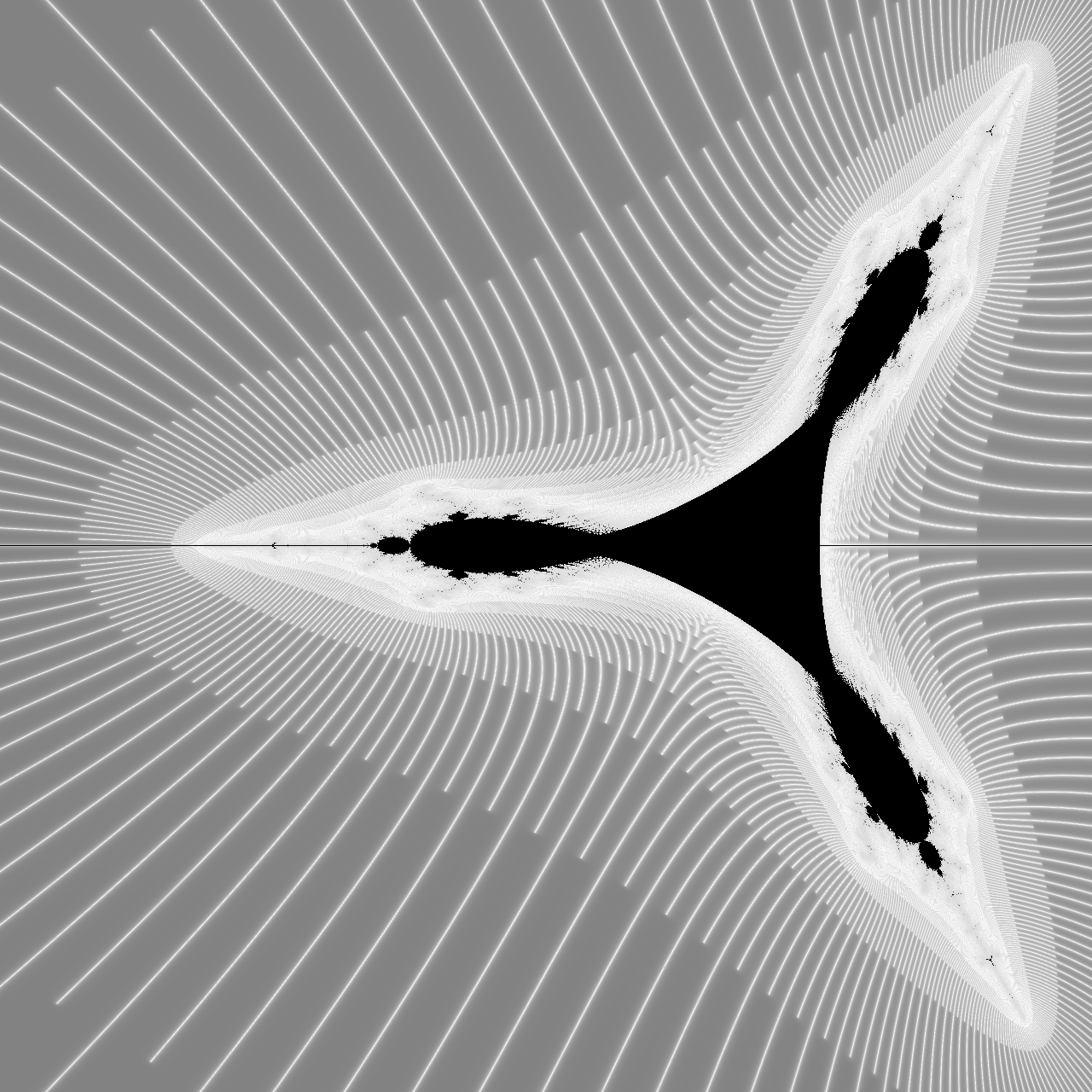}};
    \coordinate (O) at (0,0);
    \coordinate (A) at (-0.25*0.5/4*5,{0.25*sqrt(3)/2/4*5})
    edge [pre,white] (O);
    \fill [red] (A) circle (2pt);
  \end{tikzpicture}
  \includegraphics[width=5cm]{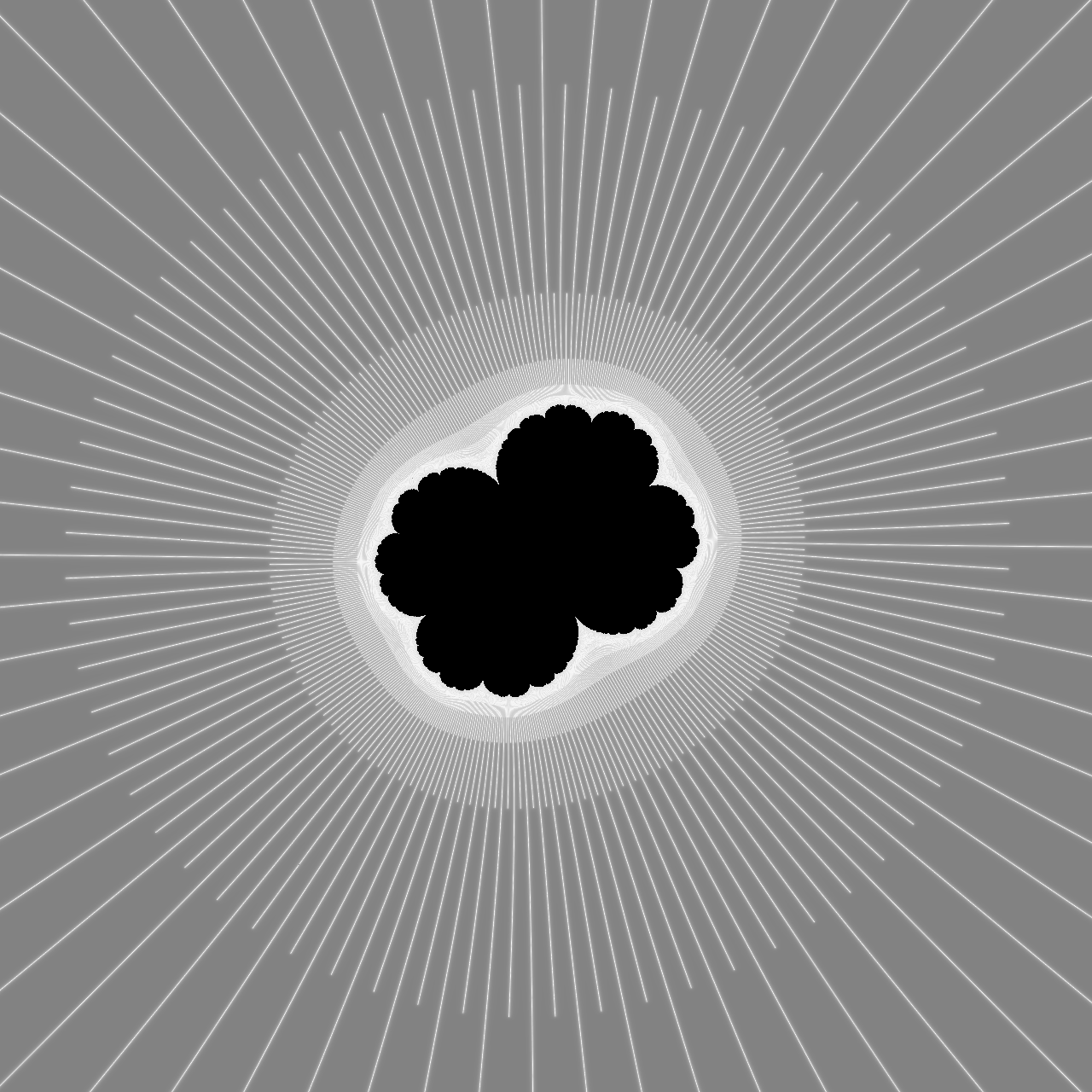}

  \usetikzlibrary{arrows.meta}
  \begin{tikzpicture}[every node/.style={inner sep=0,outer sep=0},
    pre/.style={<-,shorten <=2pt,>={Stealth[round]},semithick}]
    \node at (-0.75/4*5,0) {\includegraphics[width=5cm]{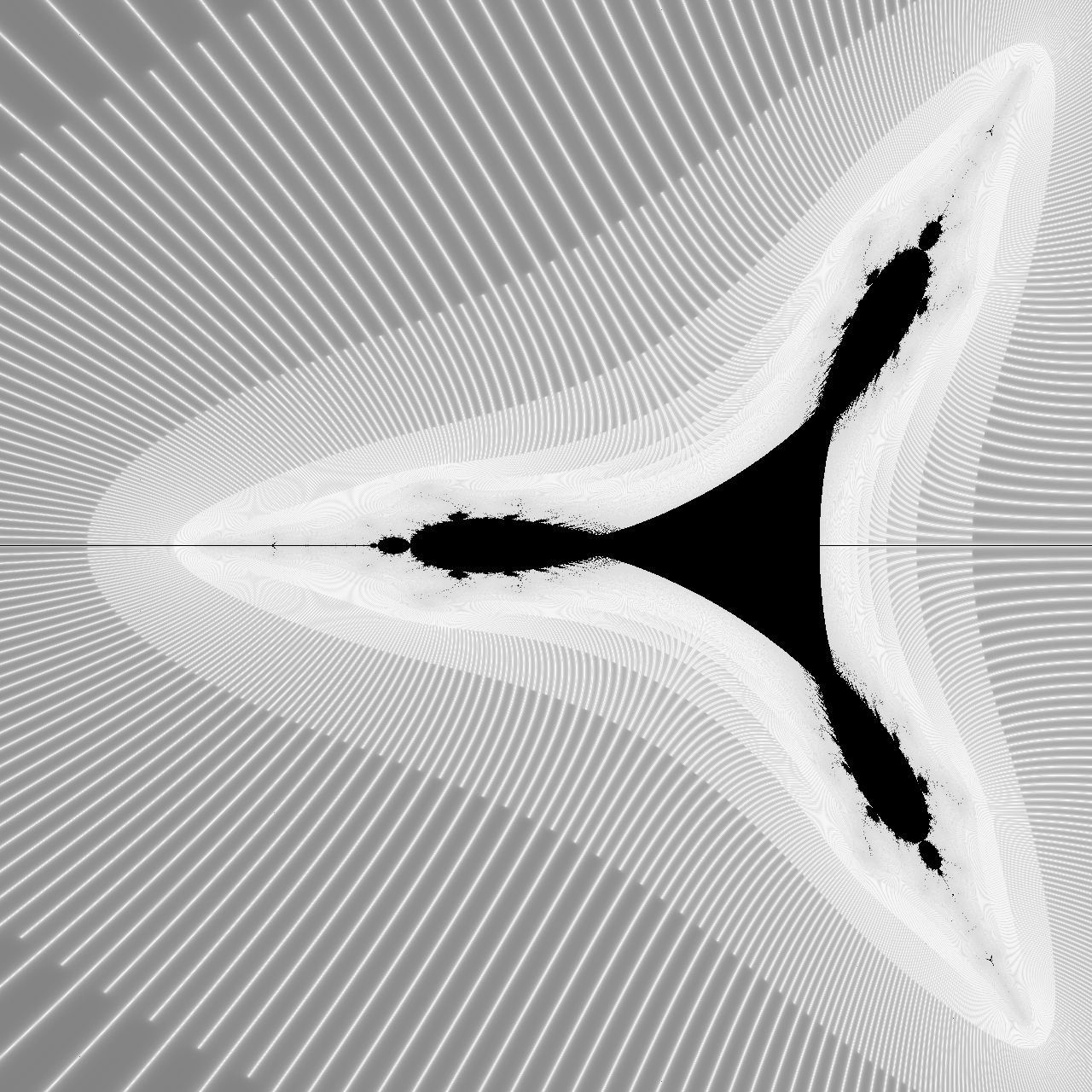}};
    \coordinate (O) at (0,0);
    \coordinate (A) at (-0.25*0.5/4*5,{0.25*sqrt(3)/2/4*5})
    edge [pre,white] (O);
    \fill [red] (A) circle (2pt);
  \end{tikzpicture}
  \includegraphics[width=5cm]{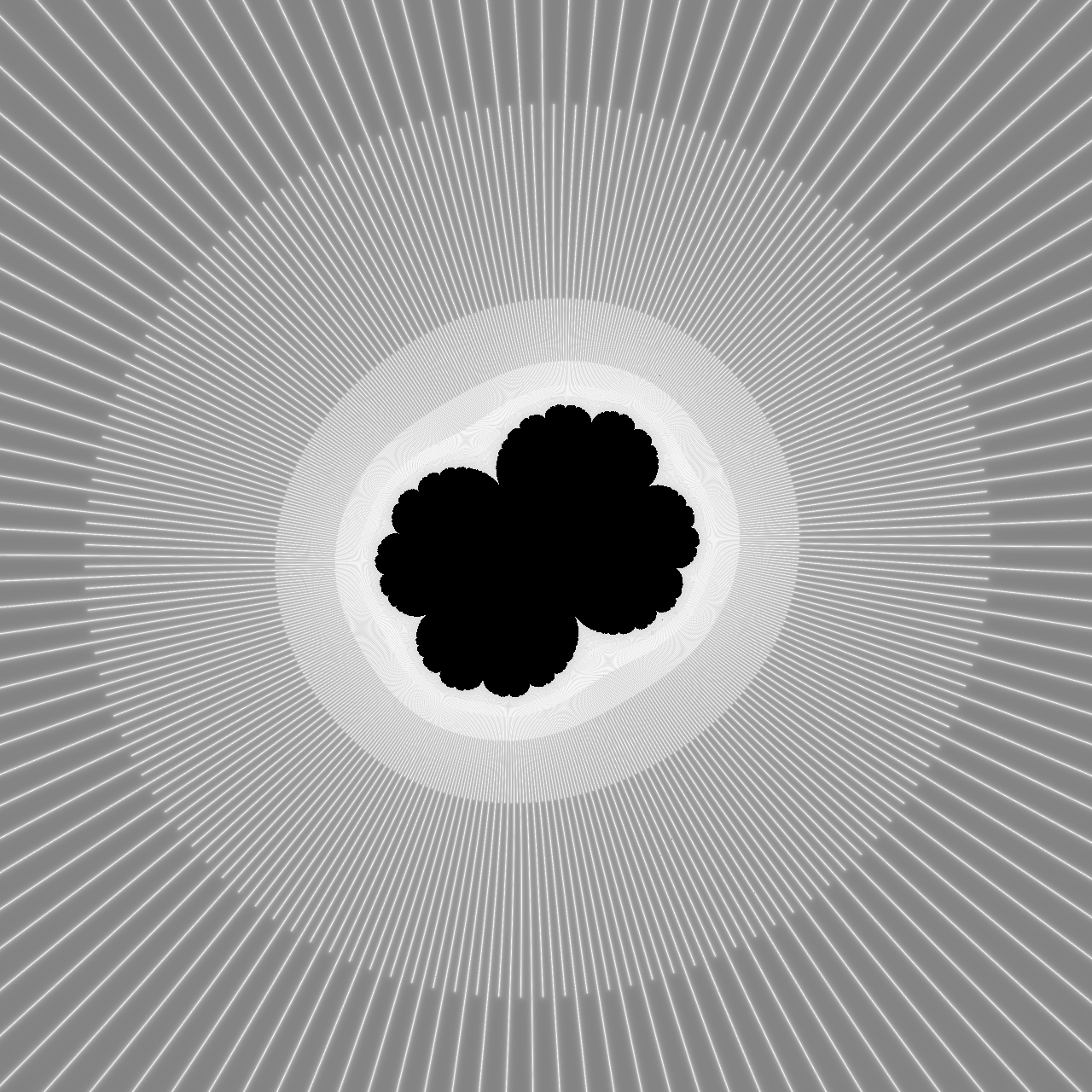}

  \caption{Decorations of baby tricorn-like sets (rescaled to 5:9) and corresponding Julia sets at parabolic parameters (indicated by the dots and arrows) for period 15 (top), 55 (middle), 105 (bottom).}
  \label{fig:decorations}
\end{figure}
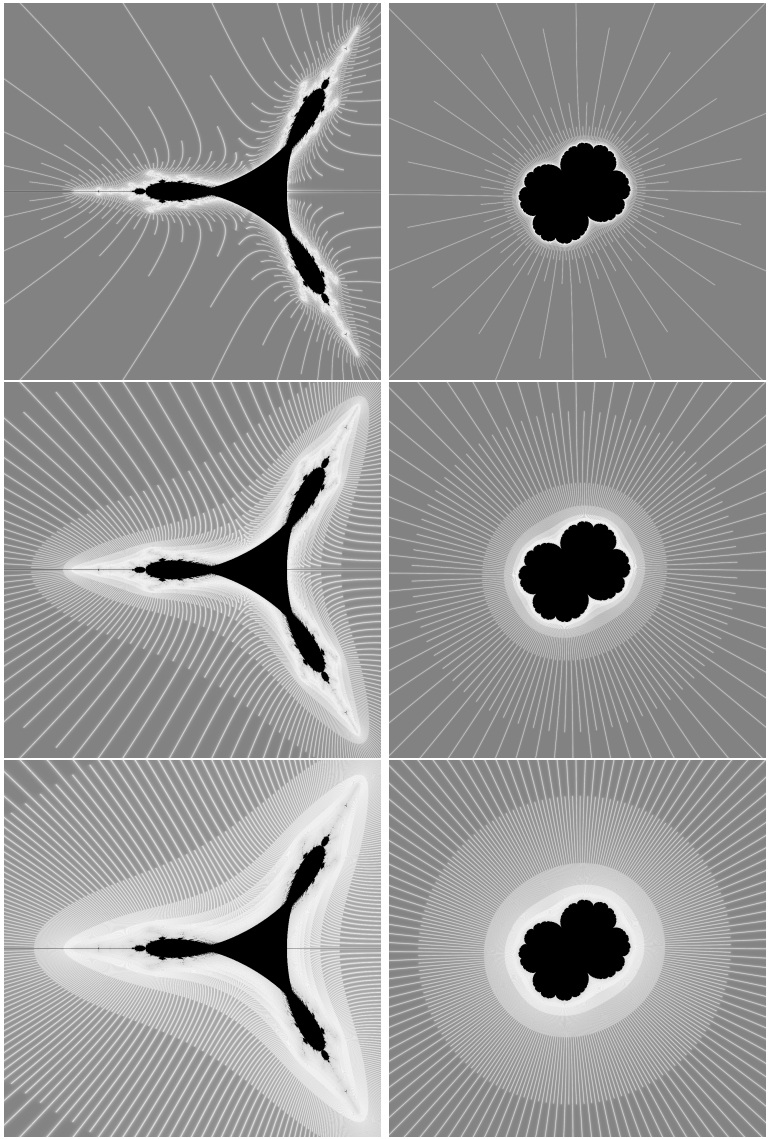
We in fact use such a convergence in the phase space to prove the accessibility.


The structure of the paper is as follows: In Section~\ref{sec:preliminaries},
we recall basic facts on the dynamics of anti-holomorphic quadratic polynomials,
especially on parabolic maps and their bifurcations,
and using Fatou coordinate, we prove a sufficient condition for accessibility (Theorem~\ref{thm:accessibility}) in terms of the \emph{critical Ecalle height}.

In Section~\ref{sec:scaling limit}, we prove that there exists a sequence of critically finite parameters $c_n \to \hat{c} = -2$, such that up to scaling,
the family $\{f_{c_n+t}^{2n+3}(z)\}$ converges to the original family $\{f_t(z)\}$ locally uniformly (Theorem~\ref{thm:scaling limit}) on the whole complex plane.
This is a simple modification of the classical fact for the holomorphic case (see \cite{MR810493}, \cite{MR816367} and \cite{MR1765082}).

To show the exponential convergence of decorations to truncated dyadic rays as $c_n \to \hat{c}$,
we need a concrete construction for hybrid conjugacies of renormalizations;
we give precise constructions of polynomial-like restrictions in Section~\ref{sec:poly-like} and hybrid conjugacies in Section~\ref{sec:straightening}.
In Section~\ref{sec:qc estimate}, we show the hybrid conjugacies just constructed converge to the identity as $n \to \infty$ exponentially fast.
In Section~\ref{sec:Boettcher}, we define a ``quasiconformal B\"ottcher coordinate'' for renormalizations and show that it converges exponentially to the identity in an fundamental annulus.
In Section~\ref{sec:approx J} we apply a result by Rivera-Letelier \cite{MR1880905} on convergence of filled Julia sets at semihyperbolic map to show
that the decorations of the small filled Julia set
converge to (truncated) dyadic dynamical rays exponentially fast.

In Section~\ref{sec:Fatou estimate}, we show the exponential convergence of the Fatou coordinate for a specific parameter,
which is necessary to check the sufficient condition in Theorem~\ref{thm:accessibility}.

Finally in Section~\ref{sec:proof}, we gather all the estimates to show the accessibility.
The key fact is that an escaping orbit grows super-exponentially for renormalizations (quadratic-like maps).
After rescaling, the renormalization hybrid equivalent to a given quadratic polynomial has a domain that grows exponentially, and converges exponentially to the polynomial.
This implies that for a given point, the escape time from the region of renormalization grows logarithmically, hence the argument (or the external angle) expands only linearly. Therefore, this does not affect exponential convergence
(see Lemma~\ref{lem:argument} and Corollary~\ref{cor:argument bound} for more detail).

\begin{ackn}
  The authors are grateful to the anonymous referees for their thoughtful comments and suggestions.
  The first author is partially supported by JSPS KAKENHI Grant Numbers JP26400115 and JP18K03367.
  The second author is partially supported by JSPS KAKENHI Grant Number JP19K03535.
\end{ackn}

\section{Preliminaries}
\label{sec:preliminaries}

The purpose of this section is to recall some basic notions and facts on anti-holomorphic dynamics, and to give a sufficient condition for accessibility (Theorem~\ref{thm:accessibility})
by discussing accessibility of hyperbolic components of odd period and parabolic bifurcations in anti-holomorphic dynamics.

\subsection{The tricorn}

We denote $K_c = K(f_c)$ be the \emph{filled Julia set} of $f_c$ and $J_c = \partial K_c$ be the \emph{Julia set} of $f_c$.

Let $\Boettcher_c$ be the B\"ottcher coordinate for $f_c$; i.e.,
the conformal isomorphism defined near infinity such that
\begin{align*}
  \Boettcher_c(f_c(z)) & = \overline{\Boettcher_c(z)}^2,
                       &
  \lim_{z \to \infty} \frac{\Boettcher_c(z)}{z} = 1.
\end{align*}
Just as the same as the holomorphic case, we can extend the domain of definition of the B\"ottcher coordinate by the functional equation above,
and we may assume the following:
\begin{itemize}
  \item $\Boettcher_c: (\C \setminus K_c) \to (\C \setminus \overline{\D})$ is a well-defined conformal isomorphism  when $c \in \cM^*$,
        where $\D$ is the unit disk;
  \item $\Boettcher_c(c)$ is well-defined when $c \not \in \cM^*$.
\end{itemize}

\begin{thm}[Nakane \cite{MR1235477}]
  The map
  \[
    \begin{array}{rccc}
      \Phi^*: & (\C \setminus \cM^*) & \to     & (\C \setminus \overline{\D}) \\
              & c                  & \mapsto & \Boettcher_{f_c}(c)
    \end{array}
  \]
  is a real-analytic diffeomorphism.
  In particular, $\cM^*$ is connected and full.
\end{thm}

\begin{defn}
  A map $f_c$ is \emph{hyperbolic} if $c \not\in \cM^*$ or $f_c$ has an attracting cycle.

  A \emph{hyperbolic component} is a maximal connected open set of $c \in \C$ for which $f_c$ has an attracting cycle.
  The \emph{period} of a hyperbolic component is defined by that of the attracting cycle of $f_c$ for any (hence all) $c$ in the hyperbolic component.
\end{defn}

For the holomorphic quadratic family, a hyperbolic component is a component of the interior of $\cM$.
On the other hand, an interior component of $\cM^*$ can contain several hyperbolic components \cite{MR1020441}.

\subsection{External rays}

\begin{defn}[Dynamical rays and parameter rays]
  For $\theta \in \R/\Z$, let
  \begin{align*}
    R_c(\theta) & := \Boettcher_c^{-1}(\{r e^{2\pi i \theta}\mid r >1\}),
                &
    \cR(\theta) & := (\Phi^*)^{-1}(\{r e^{2\pi i \theta}\mid r >1\}).
  \end{align*}
  We call $R_c(\theta)$ the \emph{dynamical ray} of angle $\theta$ for $f_c$ and
  $\cR(\theta)$ the \emph{parameter ray} of angle $\theta$.
\end{defn}

Since $\Boettcher_c$ is the B\"ottcher coordinate for the holomorphic polynomial $f_c^2$, every dynamical ray of rational angle lands at a point unless it bifurcates \cite{MR762431}.
Furthermore, every parameter ray of rational angle for the Mandelbrot set lands at a point. On the contrary, there are many parameter rays of rational angles for the tricorn which do not land:
\begin{thm}[Inou-Mukherjee \cite{MR3502067}]
  \label{thm:non-landing parameter rays}
  Any parameter ray for $\cM^*$ accumulating to a hyperbolic component of odd period greater than one does not land at a point.
  The angle of every such a ray is rational.
\end{thm}

\subsection{Accessibility}

\begin{defn}
  A point $c_0 \in \partial \cM^*$ is \emph{accessible}
  if there exists a continuous arc $c=c(s): [0,\delta] \to \C$ for some $\delta>0$
  such that $c(0)=c_0$ and $c(s) \not \in \cM^*$ for $0<s\le 1$.
  We call such $c(s)$ a \emph{path to $c$}.
  A hyperbolic component is \emph{accessible} if it has an accessible boundary point.

  A point or a hyperbolic component is \emph{inaccessible} if it is not accessible.
\end{defn}

Every hyperbolic component $\cH$ of even period is accessible,
since there are external rays landing on the boundary of $\cH$
\cite[Lemma 7.3, Lemma 7.4]{Mukherjee:2014ab}.
Hence we need only consider hyperbolic components of odd period.

By Theorem~\ref{thm:non-landing parameter rays},
we cannot prove accessibility for a hyperbolic component of odd period greater than one by showing a parameter ray land at a boundary point.
However, it is not difficult to see that the hyperbolic components of period three is still accessible (see Figure~\ref{fig:accessible}):
\begin{thm}[Inou-Mukherjee \cite{MR3502067}]
  Every hyperbolic components of period one or three contains undecorated sub-arcs.
  In particular, it is accessible.
\end{thm}
We say an arc $\gamma \subset \partial \mathcal{H} \cap \partial \cM^*$ is \emph{undecorated} if there exists a neighborhood $\mathcal{U}$ of $\gamma$ such that
$\mathcal{U} \cap \partial \cM^* = \gamma$.
In other words,
\[
  \mathcal{U} \cap \cM^* \subset \overline{\mathcal{H}},
\]
hence no other parts (``decorations'') of $\cM^*$ than $\overline{\mathcal{H}}$ intersect $\mathcal{U}$. Undecorated arcs are sometimes called ``\emph{open beaches}'' (see also Theorem~\ref{thm:accessibility} below).

\subsection{Parabolic arcs}

Here we recall some basic facts on hyperbolic components of odd period.
For further detail, see \cite{MR2020986} and \cite{Mukherjee:2014ab}.

The \emph{multiplicity} of a periodic point $z$ for $f_c$ is, by definition, that of $z$ for $f_c^2$. More precisely, let $p$ be the period of $z$ for $f_c$. Then the multiplier $\lambda$ is defined by
\[
  \lambda =
  \begin{cases}
  (f_c^p)'(z) & \text{for $p$ even,} \\
  (f_c^{2p})'(z) & \text{for $p$ odd.}
  \end{cases}
\]
A simple but important fact is that when $p$ is odd, $\lambda=|\frac{\partial}{\partial\bar{z}}(f_c^p)(z)|^2$ is a non-negative real number.

Let $\cH$ be a hyperbolic component of odd period $p\ge 3$.
Then by the above fact, for any $c \in  \partial \cH$,
there exists a periodic point $z$ such that $(f_c^{2p})'(z) = 1$.
Moreover, one can prove $z$ has the exact period $p$.

We say a parameter $c \in \partial \cH$
is a \emph{cusp} if $z$ is a double parabolic point,
i.e., $z$ has two attracting petals.
Otherwise, $z$ has only one attracting petal and we say $c$ is \emph{non-cusp}.
The boundary $\partial \cH$ consists of three cusps and three arcs connecting them.
These three arcs are called \emph{parabolic arcs}.
There exists a unique parabolic arc $\gamma \subset \partial \cH$ such that
the parabolic periodic points disconnect the filled Julia set $K_c$ for $c \in \gamma$.
We call $\gamma$ the \emph{root} arc; the others are called the \emph{co-root} arcs.
Roughly speaking, the root arc is the arc which the ``umbilical cord'' accumulates (see Figure~\ref{fig:umbilical cord}).
\begin{figure}
  \includegraphics[width=10cm]{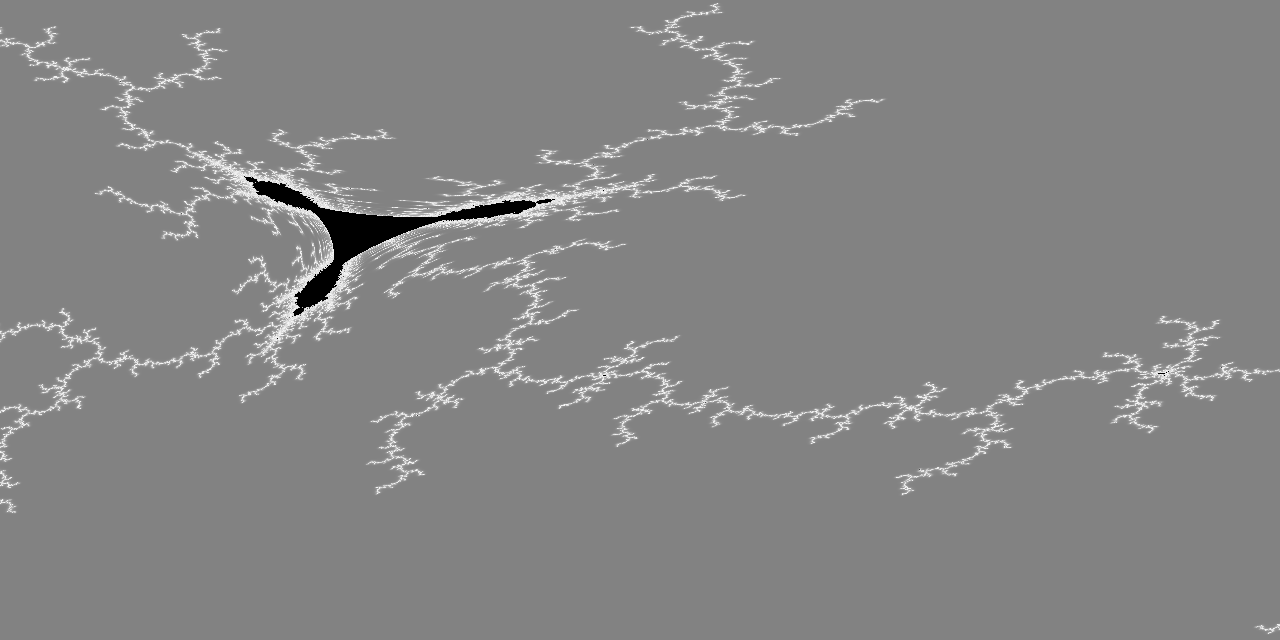}

  \caption{A hyperbolic component of odd period and its umbilical cord.}
  \label{fig:umbilical cord}
\end{figure}
Cusps are contained in the interior of the tricorn,
hence $\cH$ can be accessible only at points in (the middle of) the parabolic arcs.

We are mainly interested in accessibility on co-root arcs,
since there are less decorations that accumulate to them
(precisely speaking, there is no umbilical cord).

\subsection{Fatou coordinates}

Recall that all parameters on the boundary of a given hyperbolic component of odd period is parabolic.
Hence it is important to study parabolic bifurcation of anti-holomorphic dynamics
to discuss the accessibility of a given hyperbolic component of odd period.

Here we briefly recall the notion and basic properties of Fatou coordinates for anti-holomorphic parabolic maps and their perturbations.
Fatou coordinates for an anti-holomorphic map are ones for the holomorphic map $f^2$,
which is studied by Shishikura \cite{MR1765097}.
Holomorphic Fatou coordinates are unique up to (complex) translation, but
we further require Fatou coordinates to respect the anti-holomorphic dynamics, hence they are determined only up to \emph{real} translation.
For more details see \cite{hubbard_multicorns_2014} (see also \cite{MR3502067}).

Let $\gamma \subset \partial \cH$ be a parabolic arc of a hyperbolic component $\cH$ of odd period $p \ge 3$.
For $c \in \gamma$, let $x = x_c$ be the parabolic periodic point for $f_c$ whose immediate basin contains the critical value $c$.
Let $\Fatou_{c,\attr}$ and $\Fatou_{c,\rep}$ be attracting and repelling Fatou coordinates;
i.e., univalent maps defined on domains whose boundaries contain $x_c$
satisfying the ``anti-holomorphic Abel equation''
\begin{equation}
  \label{eqn:Abel}
  \Fatou_{c,*}(f_c(z)) = \overline{\Fatou_{c,*}(z)}+\frac{1}{2}
  \quad (*=\attr,\rep)
\end{equation}
where both sides are defined,
and the range of $\Fatou_{c,\attr}$ (resp.\ $\Fatou_{c,\rep}$)
contains a right (resp.\ left) half plane \cite[Lemma~2.3]{hubbard_multicorns_2014}.
Let us denote by $V_{c,*}$ the domain of definition of $\Fatou_{c,*}$.

Fatou coordinates are unique up to \emph{real} translation.
Hence it follows that the imaginary part $E_{c,*}(z) = \im \Fatou_{c,*}(z)$
is well-defined.
We call $E_{c,*}(z)$ the \emph{Ecalle height} of a point $z$
in the domain of $\Fatou_{c,*}$.

We can extend $\Fatou_{c,\attr}$ holomorphically to the basin of attraction of $x_c$, and $\FatouInv_c := \Fatou_{c,\rep}^{-1}$ can also be extended holomorphically to the whole complex plane.
In particular, the \emph{critical Ecalle height} $E_{c,\attr}(c)$ is well-defined.

The following simple fact is used later:
\begin{lem}
  \label{lem:c=1/4}
  For $c=\frac{1}{4}$, the external ray $R_{1/4}(0)$ is equal to
  $(\frac{1}{2},+\infty)$, and
  \[
    \FatouInv_{1/4}(R_{1/4}(0)) \subset \R.
  \]
\end{lem}

Let $c_0 \in \gamma$ and consider a perturbation $c$ of $c_0$ with $c \not \in \overline{\cH}$.
There still exist attracting and repelling Fatou coordinates $\Fatou_{c,*}:V_{c,*} \to \C$.
They are again unique up to real translation;
so we normalize them so that $\Fatou_{c,*} \to \Fatou_{c_0,*}$ as $c \to c_0$.
Moreover, the forward orbit of any point $z \in \Fatou_{c,\attr}$ eventually escapes from $V_{c,\attr}$ and enters in $V_{c,\rep}$.
This induces an isomorphism between the \emph{Ecalle cylinders}, i.e., an isomorphism from $V_{c,\attr}/f_c$ to $V_{c,\rep}/f_c$. Since Ecalle cylinders are isomorphic to $\C/\Z$ by $\Fatou_{c,*}$, such an isomorphism is in fact a translation;
in other words, we have
\[
  \Fatou_{c,\rep}(f_c^{2n}(z))-\Fatou_{c,\attr}(z) = n + C
  \quad \text{if }z \in V_{c,\attr} \text{ and } f_c^{2n}(z) \in V_{c,\rep}.
\]
for some constant $C$.
The ``anti-holomorphic Abel equation'' \eqref{eqn:Abel} implies
$C \in \R$.
This constant is called the \emph{lifted phase}.

\subsection{Accessible parameters and Fatou coordinates}

Now we relate the parabolic bifurcation theory to the accessibility of hyperbolic components of odd period.

To prove the main theorem,
we use a sufficient condition for a parameter $c_0$ in a parabolic arc
to lie in an undecorated arc, which is the following:
\begin{thm}[A sufficient condition for accessibility]
  \label{thm:accessibility}
  Let $\cH$ be a hyperbolic component of odd period
  and $c_0 \in \partial \cH$ be a non-cusp point.
  If there exists $\varepsilon>0$ such that
  \[
    \{z \in V_{c_0,\rep}|\ |E_{c_0,\rep}(z) - E_{c_0,\attr}(c_0)| < \varepsilon\}
  \]
  is contained in the basin of infinity,
  then there exists a neighborhood $\cU$ of $c_0$ such that
  $\cU \cap \cM^* = \cU \cap \overline{\cH}$.

  In particular, $c_0$ is accessible.
\end{thm}

\begin{proof}
  Take a small neighborhood $\cU$ such that for any $c \in \cU$,
  \begin{enumerate}
    \item $\{|E_{c,\rep}(z) - E_{c_0,\attr}(c_0)| < \varepsilon/2\}$ is contained in the basin of infinity for $f_c$, and
    \item $|E_{c,\attr}(c)-E_{c_0,\attr(c_0)}| < \varepsilon/2$.
  \end{enumerate}
  If further $c \not \in \overline{\cH}$,
  there exists some $n>0$ such that $f_c^n(c) \in V_{c,\rep}$.
  Since the lifted phase is real,
  $E_{c,\rep}(f_c^n(c)) = E_{c,\attr}(c)$,
  hence $c$ is in the basin of infinity and $c \not \in \cM^*$.
\end{proof}

\section{Scaling limit}
\label{sec:scaling limit}

To find parameters satisfying the assumption of Theorem~\ref{thm:accessibility},
we first specify a sequence of hyperbolic components and discuss the limiting behavior nearby both in the phase space and the parameter space.

In this section, we generalize a traditional result
on convergence of some scaling limits to the original quadratic family
(see \cite{MR816367}, \cite{MR810493} and \cite{MR1765082}).
For simplicity, we consider the following specific situation:

Let $n$ be an integer and set $N=2n+3$.
Define a sequence of parameters $\{c_n\}$ as follows:
\begin{equation}
  \label{eqn:c_n}
  c_n = \min \{c \in \R \mid Q_c^N(0)=0\}.
\end{equation}
Namely, $c_n$ is the unique real parameter such that
\begin{equation}
  \label{eqn:critical orbit}
  c_n=Q_{c_n}(0) < 0=Q_{c_n}^N(0) < Q_{c_n}^{N-1}(0) < \dots
  < Q_{c_n}^2(0).
\end{equation}
Note that \eqref{eqn:c_n} and \eqref{eqn:critical orbit} also hold for $f_c$
instead of $Q_c$ since everything is real (Figure \ref{fig:c_n}).

\begin{figure}[htbp]
  \begin{center}
    \includegraphics[width=.9\textwidth]{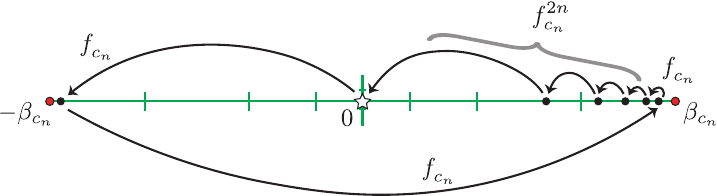}
  \end{center}
  \caption{The critical orbit of $f_{c_n}$ (or $Q_{c_n})$}
  \label{fig:c_n}
\end{figure}

In the following, we consider the case $n$ large and $c \in \C$ is close to $c_n$. In particular, $c$ is close to $\hat{c} := \lim c_n = -2$.
Let $\beta_c$ be the $\beta$-fixed point for $f_c$, i.e., the landing point of $R_c(0)$ and let $\lambda_c = (f_c^2)'(\beta_c)$ be its multiplier.
Let $u=\Koenigs_c(z)$ be a linearizing coordinate (also known as a Koenigs coordinate \cite{MR1508749}) at $\beta_c$ for $f_c$, i.e., a holomorphic map defined near $\beta_c$ such that
\begin{align}
  \label{eqn:linearizing}
  \Koenigs_c(f_c^2(w)) & = \lambda_c \Koenigs_c(w), &
  \Koenigs_c(\beta_c)  & = 0.
\end{align}
We may assume $\Koenigs_c(0)$ is defined;
we normalize $\Koenigs_c$ so that $\Koenigs_c(0)=1$, hence
\[
  u = \Koenigs_c(z) = 1+a_c z + O(z^2) \quad \text{near }z=0
\]
with $a_c \ne 0$.
The Poincar\'e map%
\footnote{The inverse of the Koenigs function was considered by Poincar\'e \cite{JMPA_1890_4_6__313_0}, as a function $\varphi(z)$ such that $\varphi(mz)$ is a rational function of $\varphi(z)$. See also \cite{lattes_sur_1918}, \cite{MR1501186} and \cite{MR2348953}. }
$\KoenigsInv_c := \Koenigs_c^{-1}$ satisfies
\[
  z = \KoenigsInv_c(u) = a_c^{-1}(u-1) + O((u-1)^2)
  \quad \text{near }u=1.
\]

Now we consider the first return map $f_c^N$ near $0$.
We factor $f_c^N = (f_c^2)^n \circ f_c^3$; that is, first $f_c^3$ anti-holomorphically maps a small neighborhood of $0$ into a neighborhood of $\beta_c$, which is easy to estimate because the number of iteration is small,
and then $(f_c^2)^N$ holomorphically maps it to a neighborhood of $0$,
which can be estimated in terms of the linearizing coordinate.

Let us denote
\begin{equation}
  \label{eqn:A_i}
  \begin{aligned}
    \Koenigs_c \circ f_c^3(z) & = \sum_{i=0}^\infty A_i(c) \bar{z}^i                     \\
                              & = A_0(c) +A_2(\hat{c}) \bar{z}^2 (1+O(|z|+|c-\hat{c}|)).
  \end{aligned}
\end{equation}
Then each $A_i(c)$ is a real-analytic function on $c$ ($A_i(c)$ contains terms in both $c$ and $\bar{c}$) defined near $\hat{c}$ and $A_2(\hat{c}) \ne 0$ (indeed, we have $A_2(\hat{c})=\frac{64}{\pi^2}>0$.)
\begin{lem}
  \label{lem:B and B*}
  There exist real-analytic functions $B(c)$ and $B^*(c)$ defined near $\hat{c}$
  such that
  \[
    A_0(c) = (c-\hat{c})B(c) + \overline{(c-\hat{c})}B^*(c).
  \]
  Moreover,
  \begin{align*}
    b_0 & := B(\hat{c}) =  \frac{2^7 \cdot 7}{15\pi^2}, &
    b_0^* := B^*(\hat{c}) = -\frac{2^8}{15\pi^2}.
  \end{align*}
  In particular, $|b_0| \ne |b_0^*|$.
\end{lem}

\begin{proof}
  Let $v_c := f_c^3(0)$ be such that
  $A_0(c) = \Koenigs_c \circ f_c^3(0) = \Koenigs_c(v_c)$.
  Note that $v_c$ is real analytic with $v_c \to v_{\hat{c}} = \beta_{\hat{c}}=2$ as $c \to \hat{c}=-2$.
  Since $\Koenigs_c(\beta_c)=0$ and $\Koenigs_c'(\beta_c) \ne 0$ for any $c$ near $\hat{c}$, we have
  \[
    \Koenigs_c(z) = \Koenigs_c'(\beta_c)(z-\beta_c) + O((z-\beta_c)^2)
  \]
  near $\beta_c$ and thus
  \[
    A_0(c) = \Koenigs_c(v_c) = \Koenigs_c'(\beta_c)(v_c - \beta_c)
    + O((v_c - \beta_c)^2).
  \]
  Since $v_c=(c^2+\bar{c})^2+c$, $\beta_c = f_c(\beta_c) = \overline{\beta_c}^2+c$ and $\beta_c = f_c^2(\beta_c)=(\beta_c^2+\bar{c})^2+c$,
  we have
  \begin{align*}
    \frac{\partial v_c}{\partial c}                                     & = 4c(c^2+\bar{c})+1, &
    \frac{\partial v_c}{\partial \bar{c}}                               & = 2(c^2+\bar{c}),      \\
    \frac{\partial \beta_c}{\partial c}                                 & =
    4\beta_c\frac{\partial \beta_c}{\partial c} \overline{\beta_c} + 1, &
    \frac{\partial \beta_c}{\partial \bar{c}}                           & =
    2(2\beta_c\frac{\partial \beta_c}{\partial \bar{c}}+1) \overline{\beta_c},
  \end{align*}
  Thus
  \begin{align*}
    \frac{\partial \beta_c}{\partial c}       & =
    -\frac{1}{4|\beta_c|^2-1},                &
    \frac{\partial \beta_c}{\partial \bar{c}} & =
    -\frac{2\overline{\beta_c}}{4|\beta_c|^2-1},
  \end{align*}
  and
  \begin{align*}
    \left. \frac{\partial}{\partial c}(v_c - \beta_c)\right|_{c=\hat{c}}
     & = -15 + \frac{1}{15} = -\frac{224}{15}, \\
    \left. \frac{\partial}{\partial \bar{c}}(v_c - \beta_c)\right|_{c=\hat{c}}
     & = 4 + \frac{4}{15} = \frac{64}{15}.
  \end{align*}
  Hence we have
  \[
    v_c - \beta_c = \frac{32}{15} (-7(c-\hat{c}) + 2\overline{(c-\hat{c})}) + O((c-\hat{c})^2).
  \]
  Since $\Koenigs_c'(\beta_c)$ is a real-analytic function of the form
  $\Koenigs_c'(\beta_c) = \Koenigs_{\hat{c}}'(\beta_c) + O(|c-\hat{c})$ near $c=\hat{c}$, we conclude that the functions $B(c)$ and $B^*(c)$ in the statement exist.

  Moreover, since $f_{\hat{c}}^2(z) = Q_{-2}^2(z)$ is the Chebyshev map of degree 4, we have
  \begin{equation}
    \label{eqn:Poincare for Chebyshev}
    \KoenigsInv_{\hat{c}}(w) = 2\cos\left(\frac{\pi}{2}\sqrt{w}\right).
  \end{equation}
  Hence
  \begin{align*}
    A_0(c) & = \Koenigs_{\hat{c}}'(\beta_{\hat{c}})(v_c-\beta_c)
    + O((v_c-\beta_c)^2)+O(c-\hat{c})(v_c-\beta_c)                            \\
           & = \frac{32}{15\KoenigsInv_{\hat{c}}'(0)}
    (-7(c-\hat{c}) + 2\overline{(c-\hat{c})})
    + O((c-\hat{c})^2)                                                        \\
           & = -\frac{128}{15\pi^2} (-7(c-\hat{c}) + 2\overline{(c-\hat{c})})
    + O((c-\hat{c})^2).
  \end{align*}
\end{proof}

By definition, $c_n$ satisfies
\begin{align*}
  \lambda_{c_n}^n A_0(c_n)
   & = \lambda_{c_n}^n \Koenigs_{c_n}  \circ f_{c_n}^3(0)
  = \Koenigs_{c_n}(f_{c_n}^{2n+3}(0))
  \\
   & =1.
\end{align*}
Hence
\begin{align*}
  \lambda_{c_n}^{-n} & = A_0(c_n)
  = (c_n-\hat{c})B(c_n) + \overline{(c_n-\hat{c})}B^*(c_n)                                      \\
                     & = b_0(c_n-\hat{c}) + b_0^*\overline{(c_n-\hat{c})} + O(|c_n-\hat{c}|^2).
\end{align*}

\begin{lem}
  \footnote{Since $b_0$ and $b_0^*$ are real,
    we may omit to take complex conjugates here.
    However, we keep taking complex conjugate
    in order to make the difference between the arguments for holomorphic and anti-holomorphic cases clear.}
  \[
    c_n = \hat{c} + \frac{\overline{b_0}\lambda_{c_n}^{-n} - b_0^* \overline{\lambda_{c_n}^{-n}}}{|b_0|^2 - |b_0^*|^2} + o(\lambda_{c_n}^{-n}).
  \]
  In particular, $c_n -\hat{c} = O(\lambda_{c_n}^{-n})$.
\end{lem}

\begin{proof}
  It just follows by solving the following system of equations:
  \[
    \begin{cases}
      \lambda_{c_n}^{-n}
      = b_0(c_n-\hat{c}) + b_0^* \overline{(c_n-\hat{c})} + O(|c_n-\hat{c}|^2), \\
      \overline{\lambda_{c_n}^{-n}}
      = \overline{b_0^*}(c_n-\hat{c}) + \overline{b_0}\overline{(c_n-\hat{c})} + O(|c_n-\hat{c}|^2).
    \end{cases}
  \]
\end{proof}

Let $\Lambda = |\lambda_{\hat{c}}| (=16>1)$.
In this section, we prove the following:
\begin{thm}
  \label{thm:scaling limit}
  Let $\frac{3}{4}<\delta < 1$ and $\frac{1}{2}<\gamma < 1$.
  Then for $|Z|=O(\Lambda^{(1-\delta) n})$ and $t = O(\Lambda^{(1-\gamma)n})$,
  \[
    g_{c_n+\rho_n(t)}(Z) \xrightarrow{n \to \infty}
    \bar{Z}^2 + (|b_0|^2 - |b_0^*|^2)t
  \]
  and the convergence is exponentially fast, where
  \begin{align*}
    g_c(Z)    & = g_{n,c}(z)  := \alpha_n f_c^N(Z/\alpha_n),        &
    \rho_n(t) & := k_n\overline{b_0}t - \overline{k_n}b_0^*\bar{t}, &
    k_n       & := \frac{a_{\hat{c}}}{\alpha_n\lambda_{c_n}^n},
  \end{align*}
  and $\alpha_n$ is the constant satisfying%
  \footnote{We actually use the condition $\frac{\alpha_n}{\overline{\alpha_n}^2}\frac{\lambda_{\hat{c}}^nA_2(\hat{c})}{a_{\hat{c}}}=1$.
    Hence we can just let $\alpha_n = \frac{\lambda_{\hat{c}}^n A_2(\hat{c})}{a_{\hat{c}}}(<0)$ because all constants are real.}
  \begin{align*}
    |\alpha_n| & = \left|\frac{\lambda_{\hat{c}}^n A_2(\hat{c})}{a_{\hat{c}}}\right|, &
    \arg \alpha_n = -\frac{1}{3} \arg\left(\frac{\lambda_{\hat{c}}^n A_2(\hat{c})}{a_{\hat{c}}}\right).
  \end{align*}
\end{thm}
Observe that
\begin{align*}
  \alpha_n & = O(\Lambda^n),     &
  k_n      & = O(\Lambda^{-2n}).
\end{align*}

\begin{rem}
  By the theorem, the hyperbolic component $\cH_n$ centered at $c_n$ converges to $\cH_0$, up to scaling. The asymptotic aspect ratio (the limit of the ratio of the vertical and horizontal scaling factor) is computed as follows:

  By Lemma~\ref{lem:B and B*} and the fact $k_n>0$,
  for $t = \zeta + \xi i$,
  \begin{align*}
    \rho_n(t) & = k_n\frac{2^7}{15\pi^2}(7(\zeta + \xi i) - 2(\zeta - \xi i)) \\
              & = \frac{2^7k_n}{15\pi^2}(5\zeta + 9\xi i).
  \end{align*}
  Therefore, the asymptotic aspect ratio is 9:5.
\end{rem}

Let $\gamma$ and $\delta$ as in the theorem, and
consider $c=c_n+s$ with $|\frac{s}{c_n-\hat{c}}| = O(\Lambda^{-\gamma n})$
and $z = O(\Lambda^{-\delta n})$.
Then by \eqref{eqn:A_i},
\begin{equation}
  \label{eqn:Koe f^N}
  \begin{aligned}
    \Koenigs_c \circ f_c^N(z) & = \lambda_c^n \Koenigs_c \circ f_c^3(z)                   \\
                              & = \lambda_c^n A_0(c) + \lambda_c^n A_2(\hat{c}) \bar{z}^2
    (1+O(|z|+|c-\hat{c}|)).
  \end{aligned}
\end{equation}

\begin{lem}
  \label{lem:lambda^n}
  \[
    \lambda_c^n = \lambda_{\hat{c}}^n(1+O(n|c-\hat{c}|)) = \lambda_{c_n}^n(1+O(ns)).
  \]
  In particular, $\lambda_c^n=O(\Lambda^{-n})$ and $c_n -\hat{c} = O(\Lambda^{-n})$.
\end{lem}

\begin{proof}
  Since $\lambda_c = \lambda_{\hat{c}}+O(|c-\hat{c}|)$, we have
  \begin{align*}
    n \log \lambda_c
     & = n\log \lambda_{\hat{c}} + n\log(1+O(|c-\hat{c}|)) \\
     & = n\log \lambda_{\hat{c}} + O(n|c-\hat{c}|)
  \end{align*}
  The second equality follows similarly.
\end{proof}

\begin{proof}[Proof of Theorem~\ref{thm:scaling limit}]
  By the above lemma,
  \begin{align*}
    \lambda_c^n A_0(c)
     & = \lambda_c^n(c-\hat{c}) B(c) + \lambda_c^n\overline{(c-\hat{c})}B^*(c) \\
     & = \lambda_{c_n}^n(1+O(ns))
    \left\{
    (c_n-\hat{c}+s)B(c_n+s) + \overline{(c_n-\hat{c}+s)}B^*(c_n+s)
    \right\}                                                                   \\
     & = \lambda_{c_n}^n(1+O(ns))                                              \\
     & \phantom{=} \times \left\{
    (c_n-\hat{c}+s)B(c_n)(1+O(s)) + \overline{(c_n-\hat{c}+s)}B^*(c_n)(1+O(s))
    \right\}                                                                   \\
     & = \lambda_{c_n}^n(1+O(ns))                                              \\
     & \phantom{=} \times \left\{
    (c_n-\hat{c})B(c_n) + \overline{(c_n-\hat{c})}B^*(c_n)
    + B(c_n)s + B^*(c_n)\bar{s}
    + O(|c_n-\hat{c}|s)
    \right\}                                                                   \\
     & = \lambda_{c_n}^n A_0(c_n)(1+O(ns))
    + \lambda_{c_n}^n(B(c_n)s + B^*(c_n)\bar{s})                               \\
     & \phantom{=} + \lambda_{c_n}^nO(ns^2)
    + \lambda_{c_n}^n(1+O(ns))O(|c-\hat{c}|s)                                  \\
     & = 1 + \lambda_{c_n}^n(B(c_n)s + B^*(c_n)\bar{s})
    + O(ns) + O(ns\Lambda^{-\gamma n}) + O(s)                                  \\
     & = 1 +\lambda_{c_n}^n(B(c_n)s + B^*(c_n)\bar{s})
    + O(ns).
  \end{align*}
  Hence it follows from \eqref{eqn:Koe f^N} that
  \begin{align*}
    \Koenigs_c \circ f_c^N(z) - 1
     & = \lambda_c^n (A_0(c) + A_2(\hat{c})\bar{z}^2(1+O(|z|+|c-\hat{c}|))) - 1 \\
     & = \lambda_c^n A_2(\hat{c})\bar{z}^2 (1+O(|z|+|c-\hat{c}|))
    + \lambda_{c_n}^n (B(c_n)s + B^*(c_n)\bar{s}+O(ns))                         \\
     & = \lambda_{\hat{c}}^n A_2(\hat{c})\bar{z}^2 (1+O(|z|+n|c-\hat{c}|))
    + \lambda_{c_n}^n (B(c_n)s + B^*(c_n)\bar{s}+O(ns)).
  \end{align*}
  Since we assume $z = O(\Lambda^{-\delta n})$ and $\delta<1$, we have
  \begin{align*}
    \lambda_{\hat{c}}^n \bar{z}^2(|z|+n|c-\hat{c}|)
     & = O(\Lambda^n \Lambda^{-2\delta n}(\Lambda^{-\delta n} + n\Lambda^{-n})) \\
     & = O(\Lambda^{(1-3\delta)n}).
  \end{align*}
  By the assumption, $s=O(|c-\hat{c}|\Lambda^{-\gamma n}) = O(\Lambda^{-(1+\gamma)n})$ and $\gamma<1$, it follows that
  \begin{align*}
    \Koenigs \circ f_c^N(z) - 1
     & = \lambda_{\hat{c}}^n A_2(\hat{c})\bar{z}^2
    + \lambda_{c_n}^n B(c_n) s + \lambda_{c_n}^n B^*(c_n) \bar{s}
    + O(ns) + O(\Lambda^{(1-3\delta)n})                                       \\
     & = O(\Lambda^n\Lambda^{-2\delta n})+ O(\Lambda^n\Lambda^{-(1+\gamma)n})
    +O(n\Lambda^{-(1+\gamma)n}) + O(\Lambda^{(1-3\delta)n})                   \\
     & = O(\Lambda^{(1-2\delta)n}) +O(\Lambda^{-\gamma n}).
  \end{align*}
  Since $a_c = a_{\hat{c}}+O(|c-\hat{c}|)=a_{\hat{c}}+O(\Lambda^{-n})$
  and $a_{\hat{c}} \ne 0$,
  we have
  \begin{align*}
    f_c^N(z) & = \KoenigsInv_c \circ \Koenigs_c \circ f_c^N(z)                         \\
             & = \frac{1}{a_c} \left\{
    \lambda_{\hat{c}}^n A_2(\hat{c}) \bar{z}^2
    + \lambda_{c_n}^n B(c_n)s + \lambda_{c_n}^n B^*(c_n) \bar{s}
    \right\}                                                                           \\
             & \phantom{=} + O(n\Lambda^{-(1+\gamma)n}) +O(\Lambda^{(1-3\delta)n})
    + O(\Lambda^{(2-4\delta)n}) + O(\Lambda^{-2\gamma n})                              \\
             & = \frac{\lambda_{\hat{c}}^n A_2(\hat{c})}{a_{\hat{c}}} \bar{z}^2
    + \frac{\lambda_{c_n}^n}{a_{\hat{c}}}(B(c_n)s + B^*(c_n)\bar{s})                   \\
             & \phantom{=} + O(\Lambda^{-2\gamma n}) + O(\Lambda^{(1-3\delta)n})
    + O(\Lambda^{(2-4\delta)n}) + O(\Lambda^{-2\delta n}) + O(\Lambda^{-(1+\gamma) n}) \\
             & = \frac{\lambda_{\hat{c}}^n A_2(\hat{c})}{a_{\hat{c}}} \bar{z}^2
    + \frac{\lambda_{c_n}^n}{a_{\hat{c}}}(B(c_n)s + B^*(c_n)\bar{s})
    + O(\Lambda^{-2\gamma n}) + O(\Lambda^{(2-4\delta)n}).
  \end{align*}

  Let $\alpha_n$ be as in the theorem and let $Z = \alpha_n z$.
  Then $\alpha_n = O(\Lambda^n)$ and $f_c^N$ is conjugate to
  \begin{align*}
    g_c(Z) & = \alpha_n f_c^N(Z/\alpha_n) \\
           & = \bar{Z}^2 + \alpha_n
    \left\{
    \lambda_{c_n}^n (B(c_n)s + B^*(c_n)\bar{s})
    + O(\Lambda^{-2\gamma n}) + O(\Lambda^{(2-4\delta)n})
    \right\}                              \\
           & = \bar{Z}^2 + \xi_n(s)
    + O(\Lambda^{(1-2\gamma) n}) + O(\Lambda^{(3-4\delta)n})
  \end{align*}
  where
  \[
    \xi_n(s) :=
    \frac{\alpha_n\lambda_{c_n}^n}{a_{\hat{c}}} (B(c_n)s + B^*(c_n)\bar{s}).
  \]
  Now recall that $k_n = \frac{a_{\hat{c}}}{\alpha_n\lambda_{c_n}^n}$.
  Let
  \[
    s = \tilde{\rho}_n(t)
    := k_n \overline{B(c_n)}t - \overline{k_n} B^*(c_n) \bar{t}.
  \]
  Then for $t = O(\Lambda^{(1-\gamma)n})$, we have
  \begin{align*}
    \frac{s}{c-\hat{c}} & = O(\Lambda^{-\gamma n}),
  \end{align*}
  and
  \begin{align*}
    \xi_n(s) & = \frac{\alpha_n\lambda_{c_n}^n}{a_{\hat{c}}}
    (B(c_n)(k_n \overline{B(c_n)}t - \overline{k_n} B^*(c_n) \bar{t})
    + B^*(c_n)(\overline{k_n}B(c_n) \bar{t} - k_n \overline{B^*(c_n)} t)) \\
             & = (|B(c_n)|^2 - |B^*(c_n)|^2)t
    = (|b_0|^2 - |b_0^*|^2)t - O(\Lambda^{-\gamma n}).
  \end{align*}
  Therefore, we have
  \begin{equation}
    \label{eqn:scaling limit}
    g_{c_n+\rho_n(t)}(z)
    - \left\{ \bar{z}^2 + (|b_0|^2 - |b_0^*|^2)t \right\}
    = O(\Lambda^{(3-4\delta)n}) + O(\Lambda^{(1-2\gamma)n}) + O(\Lambda^{-\gamma n}).
  \end{equation}
  Since $\gamma \in (\frac{1}{2},1)$ and $\delta \in (\frac{3}{4},1)$,
  $1-2\gamma$ and $3-4\delta$ are negative.
  Hence $g_{c_n+\rho_n(t)}$ converges exponentially to $\bar{z}^2+(|b_0|^2-|b_0^*|^2)t$.
\end{proof}

\section{Construction of polynomial-like restrictions}
\label{sec:poly-like}

We have proved local uniform convergence of a sequence of rescaled family in the previous section.
However, more detailed estimates are required because we need to regard a rescaled map as a (anti-holomorphic) quadratic-like map with large domain and range, and to discuss its exponential convergence to its straightening (i.e., the quadratic polynomial hybrid equivalent to it).

To do this end, we first construct polynomial-like restrictions concretely in this section.
In the next section, we give an explicit construction of hybrid conjugacies to its straightening, and show the exponential convergence of hybrid conjugacies to the identity in Section~6.

In the following, we consider the case $c = c_n+\rho_n(t)$ with
\[
  t = O(1)\ (=O(\Lambda^{(1-\gamma)n})\text{ for any }\gamma \in (\tfrac{1}{2},1)).
\]
For $f_{-2}$, the external rays of angles $1/3$ and $2/3$ land at the
``$\alpha$-fixed point'' $-1$. Since $-1$ is repelling,
$R_c(1/3)$ and $R_c(2/3)$ land at the same point when $n$ is sufficiently large.

We construct Yoccoz puzzles from these rays.
Let $D^0=D^0_c$ be the puzzle piece of depth $1$ containing $0$ for $f_c$.
Namely, let
\[
  \Gamma_0 = \Boettcher_c^{-1}(\{|z|=2\}) \cup \overline{R_c(1/3) \cup R_c(2/3)}
\]
and let $D^0$ be the closure of the bounded component of $\C \setminus f_c^{-1}(\Gamma_0)$ containing $0$.

Let
\[
  D^1 = D_c^1 := f_c^{-1}(-\Koenigs_c^{-1}(\lambda_c^{-n-1}\Koenigs_c(D_0))).
\]
In other words, $D^1$ is the puzzle piece of depth $N=2n+3$ containing $0$.
Hence we have the following:
\begin{lem}
  The set $D_c^1$ is the component of $f_c^{-N}(D^0)$ containing $0$ and $f_c^N:\Int D^1 \to \Int D^0$ is a quadratic-like mapping.
\end{lem}

Let $D^k_c := (f_c^N|_{D^1_c})^{-k}(D_0)$ for $k>1$.
It is easy to see from the definition that
$\diam D^1 = O(\lambda_c^{-\frac{n}{2}})$.
Therefore, $f_c^N:D^k_c \to D^{k-1}_c$ is close to a composition of a linear map
and $z^2$ for $k \ge 2$. More precisely,

\begin{lem}
  \label{lem:rough approximation}
  Fix $k \ge 1$. Then $\diam D^k_c \asymp \Lambda^{-(1-2^{-k})n}$, and for $z \in D^{k+1}_c$, we have
  \[
    f_c^N(z)
    + \frac{\lambda_{c_n}^{n+1}\Koenigs_{c_n}'(-c_n)}{\Koenigs_{c_n}'(0)}
    \bar{z}^2
    = O(\Lambda^{-n}).
  \]
\end{lem}

Here, $A \asymp B$ stands for $A=O(B)$ and $B=O(A)$.
Note that all constants can depend on $k$ (and $m$ which appears later),
but uniform on $n$,
since we will fix $m$ and consider $k \le m$, then we let $n$ tends to infinity.

\begin{proof}
  First observe that $c-c_n = \rho_n(t) = O(\Lambda^{-2n})$.
  Take $\frac{1}{2} \le a < 1$ and let
  $z = O(\Lambda^{-an})$.
  Recall that $\lambda_c^m = \lambda_{c_n}^m(1+O(m(c-c_n)))$ by Lemma~\ref{lem:lambda^n}.
  Hence it follows that
  \begin{align*}
    \Koenigs_c(-(\bar{z}^2+c))
     & = \Koenigs_c(-c_n) - \Koenigs_c'(-c_n)(\bar{z}^2+c-c_n)
    + O((\bar{z}^2+c-c_n)^2)                                       \\
     & = \Koenigs_c(-c_n) - \Koenigs_c'(-c_n)(\bar{z}^2+\rho_n(t))
    + O(\Lambda^{-4an}),
  \end{align*}
  and
  \begin{align*}
    \lambda_c^{n+1} \Koenigs_c(-c_n)
     & = \lambda_{c_n}^{n+1} \Koenigs_{c_n}(-c_n)
    + \lambda_{c_n}^{n+1}(\Koenigs_c(-c_n) - \Koenigs_{c_n}(-c_n))     \\
     & \quad + (\lambda_c^{n+1} - \lambda_{c_n}^{n+1})\Koenigs_c(-c_n) \\
     & = \lambda_{c_n}^{n+1} \Koenigs_{c_n}(-c_n)
    + \lambda_{c_n}^{n+1}O(c-c_n)
    + \lambda_{c_n}^{n+1} O((n+1)(c-c_n))O(\beta_{c_n}+c_n)            \\
     & = \lambda_{c_n}^{n+1} \Koenigs_{c_n}(-c_n)
    + O(\Lambda^{-n}) + O(n\Lambda^{-2n})                              \\
     & = \lambda_{c_n}^{n+1} \Koenigs_{c_n}(-c_n)
    +O(\Lambda^{-n}),
  \end{align*}
  so
  \begin{align*}
    \lambda_c^{n+1} \Koenigs_c(-(\bar{z}^2+c))
     & = \lambda_{c_n}^{n+1} \Koenigs_{c_n}(-c_n) + O(\Lambda^{-n})
    - \lambda_{c_n}^{n+1} \Koenigs_{c_n}'(-c_n) (\bar{z}^2+\rho_n(t)) \\
     & \quad + (\lambda_{c_n}^{n+1} \Koenigs_{c_n}'(-c_n) -
    \lambda_c^{n+1} \Koenigs_c'(-c_n))O(\Lambda^{-2an})
    + O(\Lambda^{-(4a-1)n})                                           \\
     & = \lambda_{c_n}^{n+1} \Koenigs_{c_n}(-c_n)
    - \lambda_{c_n}^{n+1} \Koenigs_{c_n}'(-c_n) (\bar{z}^2+\rho_n(t)) \\
     & \quad + O(\Lambda^{-n}) + (O(n\Lambda^{-n})
    + O(\Lambda^{-n}))O(\Lambda^{-2an}) + O(\Lambda^{-(4a-1)n})       \\
     & = \lambda_{c_n}^{n+1} \Koenigs_{c_n}(-c_n)
    - \lambda_{c_n}^{n+1} \Koenigs_{c_n}'(-c_n) (\bar{z}^2+\rho_n(t))
    + O(\Lambda^{-n}),
  \end{align*}
  and since $\KoenigsInv_{c_n}(\lambda_{c_n}^{n+1}\Koenigs_{c_n}(-c_n))
    = f_{c_n}^N(0) = 0$,
  \begin{align*}
    f_c^N(z)
     & = \KoenigsInv_c(\lambda_c^{n+1}\Koenigs_c(-(\bar{z}^2+c)))                       \\
     & = \KoenigsInv_{c_n}(\lambda_c^{n+1}\Koenigs_c(-(\bar{z}^2+c)))
    + O(\Lambda^{-2n})                                                                  \\
     & = \frac{1}{\Koenigs_{c_n}'(0)}
    (-\lambda_{c_n}^{n+1}\Koenigs_{c_n}'(-c_n)(\bar{z}^2+\rho_n(t))
    + O(\Lambda^{-n})) + O(\Lambda^{-2n})                                               \\
     & = -\frac{\lambda_{c_n}^{n+1}\Koenigs_{c_n}'(-c_n)}{\Koenigs_{c_n}'(0)} \bar{z}^2
    + O(\Lambda^{-n}) \quad (=O(\Lambda^{-(2a-1)n})).
  \end{align*}
  Therefore, if $z \in D_c^1$ and $f_c^N(z) = O(\Lambda^{-(2a-1)n})$,
  then $z = O(\Lambda^{-an})$.
  Hence $\diam D_c^k \asymp \Lambda^{-(1-2^{-k})n}$ follows by induction.
\end{proof}

\section{Straightening}
\label{sec:straightening}

Fix $m \ge 0$. Consider a polynomial-like restriction
\[
  f_c^N: D_c^{m+1} \to D_c^m.
\]
The \emph{straightening} is a quadratic anti-holomorphic polynomial
hybrid equivalent to it.

To obtain an exponential convergence to the straightening in a large domain,
here we give a precise construction of a hybrid conjugacy and the straightening,
following Douady and Hubbard \cite{MR816367}.

First we rescale as before, so let
$g_c(z) := \alpha_n f_c^N(\frac{z}{\alpha_n})$ and
$\tilde{D}_c^m := \alpha_n D_c^m$. Then
\[
  g_c:\tilde{D}_c^{m+1} \to \tilde{D}_c^m
\]
is quadratic-like.

Since we consider the case
$|\frac{c-c_n}{c_n+2}| = O(\Lambda^{-2n})/O(\Lambda^{-n}) = O(\Lambda^{-n})$
and $|z| < \diam D_c^{m+1} = O(\Lambda^{-(1-2^{-m-1})n})$,
we can apply Theorem~\ref{thm:scaling limit} and \eqref{eqn:scaling limit}
with $\gamma <1$ sufficiently close to $1$ and $\delta = 1-2^{-m-1}$ ($m \ge 2$).
Namely, we have
\begin{equation}
  \label{eqn:convergence}
  g_c(z) = \bar{z}^2 + s + O(\Lambda^{-(1-2^{-m+1})n})
\end{equation}
where $s = (|b_0|^2 - |b_0^*|^2)t$, and $g_c$ is close to $f_s(z)=\bar{z}^2+s$.
(As mentioned before, $m$ will be fixed, so constants can depend on $m$.)

Next we enlarge the domain of
$g_c:\tilde{D}_c^{m+1} \to \tilde{D}_c^m$.
\begin{lem}
  \label{lem:round disks}
  For sufficiently large $n$,
  there exist $r_m$ and $R_m$ with $0 < r_m<R_m$ such that
  \begin{enumerate}
    \item
          $r_m \asymp R_m \asymp \Lambda^{-(1-2^{-m})n}$, and
    \item
          for any $c=c_n+\rho_n(t)$ with $t=O(1)$,
          $$
            D_c^{m+1} \subset \D(r_m)
            \subset D_c^{m} \subset \D(R_m),
          $$
          where $\D(r):=\{z \in \C \mid |z|<r \}$.
  \end{enumerate}
\end{lem}
\begin{proof}
  There are $r_0$ and $R_0$ with $0 <r_0 <R_0$
  independent of $n$ such that $\partial D_c^0
    \subset \{ z \in \C \mid r_0 \le |z| \le R_0 \}$
  for any $c=c_n+\rho_n(t)$ with $t=O(1)$.
  By Lemma \ref{lem:rough approximation},
  there exists a constant $K$ such that
  for sufficiently large $n$
  and for
  $c=c_n+\rho_n(t)$ with $t=O(1)$,
  we have
  $$
    K^{-1}\Lambda^{-n/2} |f_c^N(z)|^{1/2}
    \le |z| \le
    K \Lambda^{-n/2} |f_c^N(z)|^{1/2}
  $$
  if $f_c^N(z) \in \overline{\D(R_0)} \setminus D_c^{m+1}$.
  Now the existence of $r_m$ and $R_m$ simply follows by induction.
\end{proof}

Let $R:= |\alpha_n| R_m \asymp \Lambda^{2^{-m} n}$
such that $\tilde{D}_c^m \subset \D(R)$.
Let
\begin{align*}
  U =\D(R), \qquad
  U' = U_1 := g_c^{-1}(U),
  \quad\text{and}\quad
  U_2' := f_s^{-1}(U).
\end{align*}
Then $g_c:U_1' \to U$ and $f_s:U_2' \to U$ are quadratic-like mappings.
In particular, $g_c:U_1' \to U$ still satisfies (\ref{eqn:convergence}).

Let
\begin{align*}
  \gamma_1 & :\R/\Z \to \partial U_1', &
  \gamma_2 & :\R/\Z \to \partial U_2'
\end{align*}
be such that $g_c(\gamma_1(\theta)) = f_s(\gamma_2(\theta)) = Re^{-4\pi i \theta}$ and $\arg \gamma_j(\theta)$ $(j=1,2)$ are close to $2\pi \theta$.
Let $\mathcal{A}_0 := \{z \in \C \mid \sqrt{R} \le |z| \le R\}$ and
$\mathcal{A}_j := \overline{U} \setminus U_j'$.
Define diffeomorphisms $h_j:\mathcal{A}_0 \to \mathcal{A}_j$ by
\begin{align*}
  h_j(re^{2\pi i \theta})
   & =
  \frac{R-r}{R-\sqrt{R}}\gamma_j(\theta) +
  \frac{r-\sqrt{R}}{R-\sqrt{R}} Re^{2\pi i \theta}                     \\
   & = re^{2\pi i \theta} + \frac{R-r}{R-\sqrt{R}} \gamma_j^1(\theta),
\end{align*}
where $\gamma_j^1(\theta) := \gamma_j(\theta) - \sqrt{R}e^{2\pi i \theta}$.
Let $h := h_2 \circ h_1^{-1} \colon \mathcal{A}_1 \to \mathcal{A}_2$.
Then
\begin{equation}
  \label{eqn:h}
  h(h_1(re^{2\pi i \theta}))
  =
  h_1(re^{2\pi i \theta}) +
  \frac{R-r}{R-\sqrt{R}}(\gamma_2^1(\theta) - \gamma_1^1(\theta)).
\end{equation}

Let $h(z)=z$ on $\{|z| > R\}$.
Then $h_j: \{|z| \ge \sqrt{R}\} \to (\C \setminus U_j')$ is a piecewise
smooth homeomorphism, hence quasiconformal.
Thus $h = h_2 \circ h_1^{-1}: (\C \setminus U_1') \to (\C \setminus U_2')$ is
also a quasiconformal homeomorphism.

Now define an anti-(i.e., orientation reversing) quasiregular mapping
$F_c:\C \to \C$ as follows:
\[
  F_c(z) =
  \begin{cases}
    g_c(z)            & z \in U_1',      \\
    h^{-1}(f_s(h(z))) & z \not \in U_1'.
  \end{cases}
\]
Define an almost complex structure $\sigma$ as follows:
Let $\sigma:= \sigma_0$ (the standard complex structure) on $\C \setminus U$ and on $K(g_c)$. For $z \in U \setminus K(g_c)$, let $n>0$ be the smallest integer satisfying $F_c^n(z) \not \in U$. Then define
\[
  \sigma := (F_c^{n})^*\sigma_0
\]
at $z$.
Then $\sigma$ is $F_c$-invariant. By the measurable Riemann mapping theorem,
there exists a quasiconformal homeomorphism $\eta = \eta_c:\C \to \C$ fixing $0$ and $\infty$ such that $\eta^*\sigma_0 = \sigma$ and $\eta$ is tangent to the identity at $\infty$ (note that $\eta$ is holomorphic near $\infty$).

Then $P := \eta \circ F_c \circ \eta^{-1}$ is an anti-holomorphic map and because of the normalization of $\eta$, hence $P = f_{\tilde{c}}$ for some $\tilde{c}$.
We denote $\tilde{c} = \chi_n(c)$. Note that $s \ne \chi_n(c)$ in general (see Lemma~\ref{lem:convergence 2} below).

\section{Estimate of quasiconformal homeomorphisms}
\label{sec:qc estimate}

Here we estimate how close the hybrid conjugacy $\eta = \eta_c$
defined in the previous section to the identity:

\begin{lem}
  \label{lem:qc estimate}
  The hybrid conjugacy $\eta$ satisfies
  $\eta(z)-z = O(\Lambda^{-(1-2^{-m+1})n})$ on $U$.
\end{lem}

\begin{proof}
  Let $u = u(c,z) := g_c(z)-f_s(z)$, which is anti-holomorphic in variable $z$ and $u = O(\Lambda^{-(1-2^{-m+1})n})$ by \eqref{eqn:convergence}.
  To simplify the notation, we assume $|u| \le \ell = O(\Lambda^{-(1-2^{-m+1})n})$ on $U'$.

  Let $\tau := 2\pi \theta$ with $0 \le \theta \le 1$ and
  $w(\tau) := Re^{-2\tau i}$.
  By replacing $R$ by $R/2$ if necessary, we can find an anti-holomorphic branch $z = G_c(w)$ of $g_c^{-1}$ on the disk
  $\Delta_\tau:= \{w \in \C \mid |w-w(\tau)| < R/2\}$ such that
  \[
    \bar{z} = \overline{G_c(w)} = (w-s-u)^{1/2}
    = \sqrt{w}\left(1-\frac{s+u}{w}\right)^{1/2}
    = \sqrt{w}+O\left(\frac{1}{\sqrt{w}}\right),
  \]
  where we choose the branch of $\sqrt{w}$ such that
  $\sqrt{w(\tau)} = \sqrt{R}\,e^{-\tau i}$.
  Let $V_1(w) := \overline{G_c(w)} - \sqrt{w}$ such that
  \[
    V_1(w(\tau)) =
    \overline{\gamma_1^1\left(\frac{\tau}{2\pi}\right)} = \overline{\gamma_1^1(\theta)}.
  \]
  Since $V_1(w)$ is holomorphic and $V_1(w) = O(1/R^{1/2})$ on $\Delta_\tau$, we have
  \[
    V_1'(w(\tau)) = O(1/R^{3/2}),
  \]
  hence
  \[
    \frac{d}{d \tau} V_1(w(\tau))
    =
    V_1'(w(\tau)) \cdot w'(\tau)
    = O(1/R^{1/2}).
  \]
  Similarly, we define a branch $z = F_s(w)$ of $f_s^{-1}$ on $\Delta_r$ such that
  \[
    \bar{z} = \overline{F_s(w)}
    = \sqrt{w}\left(1-\frac{s}{w}\right)^{1/2}
  \]
  Let $V_2(w) := \overline{F_s(w)} - \sqrt{w}$ such that
  \[
    V_2(w(\tau)) = \overline{\gamma_2^1\left(\frac{\tau}{2\pi}\right)}
    = \overline{\gamma_2^1(\theta)}
  \]
  with derivative
  $\dfrac{d}{d \tau} V_2(w(\tau)) = O(1/R^{1/2})$.
  Hence we have
  \[
    \Gamma(\tau) :=
    \overline{\gamma_2^1(\theta) - \gamma_1^1(\theta)}
    = V_2(w(\tau)) - V_1(w(\tau)).
  \]
  Since
  \begin{align*}
    V_2(w) - V_1(w)
     & = \overline{F_s(w) -G_c(w)}                       \\
     & = \sqrt{w}\left\{\left(1-\frac{s}{w}\right)^{1/2}
    -\left(1-\frac{s+u}{w}\right)^{1/2}\right\}          \\
     & =\sqrt{w} \,O(u/w),
  \end{align*}
  we have
  \[
    \Gamma(\tau) =\sqrt{R} \,O(\ell/R)=O(\ell/R^{1/2}).
  \]
  By the Schwarz lemma,
  \[
    \frac{d}{dw}(V_2(w)-V_1(w)) = O(\ell/R^{3/2})
  \]
  at $w=w(\tau)$, hence it follows that
  \begin{align*}
    \Gamma'(\tau)
     & = \left. \frac{d}{dw}(V_2(w)-V_1(w))\right|_{w=w(\tau)}
    \cdot w'(\tau)                                             \\
     & = O(\ell/R^{3/2}) \cdot O(R) = O(\ell/R^{1/2}).
  \end{align*}
  Let $H(z) := h(z)-z$.
  By \eqref{eqn:h}, it follows that for $z=h_1(re^{\tau i})$,
  \[
    H(z) = \frac{R-r}{R-\sqrt{R}}\Gamma(\tau).
  \]
  Then
  \begin{align*}
    H_r    & = -\frac{\Gamma(\tau)}{R-\sqrt{R}} = O(\ell/R^{3/2}),                                  \\
    H_\tau & = \frac{R-r}{R-\sqrt{R}}\cdot\Gamma'(\tau)
    = O(\ell/R^{1/2}),                                                                              \\
    z_r    & = e^{\tau i} - \frac{\gamma_1^1(\theta)}{R-\sqrt{R}} = O(1),                           \\
    z_\tau & = ire^{\tau i} + \frac{R-r}{R-\sqrt{R}}\frac{d}{d\tau} \overline{V_1(w(\tau))} = O(R),
  \end{align*}
  and similarly $\bar{z}_r = e^{-\tau i} + O(1/R^{3/2})=O(1)$
  and $\bar{z}_r = -ire^{-\tau i}+O(1/R^{1/2})=O(R)$.
  In particular,
  $z_r \bar{z}_\tau - z_\tau \bar{z}_r=-2ir+O(1/R^{1/2})$
  where $r$ ranges from $\sqrt{R}$ to $R$.
  Therefore
  \begin{align*}
    \begin{pmatrix}
      H_z & H_{\bar{z}}
    \end{pmatrix}
     & =
    \begin{pmatrix}
      H_r & H_{\tau}
    \end{pmatrix}
    \begin{pmatrix}
      z_r & z_\tau \\ \bar{z}_r & \bar{z}_\tau
    \end{pmatrix}^{-1} \\
     & = \begin{pmatrix}
           O(\ell/R^{3/2}) & O(\ell/R^{1/2})
         \end{pmatrix}
    \cdot O(1/R^{1/2})
    \begin{pmatrix}
      O(R) & O(R) \\
      O(1) & O(1)
    \end{pmatrix}                          \\
     & = \begin{pmatrix}
           O(\ell/R) & O(\ell/R)
         \end{pmatrix}
  \end{align*}
  on $\mathcal{A}_1$.
  Thus the complex dilatation $\mu_h$ of $h$ satisfies
  \[
    \mu_h = \frac{H_{\bar{z}}}{1+H_z} = O(\ell/R)
  \]
  as well. This implies that
  \begin{equation}
    \label{eqn:dilatation of eta}
    \|\mu_\eta\|_{\infty} = \|\mu_h\|_{\infty} = O(\ell/R)
    (= O(\Lambda^{-(1-2^{-m})n}))
  \end{equation}
  where $\|\cdot\|_p$ is the norm for $L^p(\C)$.
  Since $\mu_\eta$ is supported on $U \setminus K(g_c)$, it is of compact
  support. Hence by \cite[Theorem~4.24]{MR1215481} and the proof of its Corollary~2, the following holds: For given $m>0$ and $p>2$,
  if $n$ is sufficiently large, then we have
  \[
    |\eta(z)-z| \le K_p \|\mu_\eta\|_p |z|^{1-2/p}
  \]
  for any $z \in \C$, where $K_p>0$ is a constant depending only on $p$.
  Since $\mu_\eta$ is supported on $U$, we have
  \begin{align*}
    \|\mu_\eta\|_p
     & \le \|\mu_\eta\|_\infty \cdot (\area U)^{1/p} \\
     & = \|\mu_\eta\|_\infty(\pi R^2)^{1/p}          \\
     & = O(\ell\, R^{-1+2/p}).
  \end{align*}
  Hence if $z \in U$, we obtain
  \begin{align*}
    |\eta(z)-z|
     & \le O(\ell R^{-1+2/p}) \cdot R^{1-2/p} \\
     & = O(\ell)=O(\Lambda^{-(1-2^{-m+1})n}).
  \end{align*}
\end{proof}

\begin{lem}
  \label{lem:convergence 2}
  Let $\tilde{c} \in \cM^*$ and
  assume a sequence $\{b_n = c_n + \rho_n(t_n)\}$
  satisfies $\chi_n(b_n) = \tilde{c}$ for sufficiently large $n$.
  Then  $s_n = (|b_0|^2 - |b_0^*|^2)t_n \to \tilde{c}$ and
  $g_{n,b_n}$ converges locally uniformly to $f_{\tilde{c}}$.
\end{lem}

\begin{proof}

  Since the critical value of $g_{b_n}$ is $s_n + O(\Lambda^{-(1-2^{-m+1})n})$
  and it is mapped to $\tilde{c}$ by $\eta_{b_n}$,
  it follows by Lemma~\ref{lem:qc estimate} that
  \[
    \tilde{c} - \{s_n + O(\Lambda^{-(1-2^{-m+1})n})\}
    = O(\Lambda^{-(1-2^{-m+1})n}).
  \]
  Hence $s_n \to \tilde{c}$ as $n \to  \infty$ and local uniform convergence
  follows by \eqref{eqn:convergence}.
\end{proof}

\section{Estimate of B\"ottcher coordinates}
\label{sec:Boettcher}

We want to approximate the decorations attached to the filled Julia set of a given polynomial-like restriction by the truncated binary external rays of the straightening.
Hence we need to estimate B\"ottcher coordinates.

Recall that $\Boettcher_c$ is the B\"ottcher coordinate for $f_c$.
For $c = c_n+\rho_n(t)$ with $t=O(1)$,
we define a ``quasiconformal B\"ottcher coordinate''
$\Boettcher_{g_c}:(\overline{U} \setminus K(g_c)) \to \C$ for $g_c$ as follows:
First for $z \in \mathcal{A}_1 = \overline{U} \setminus U'$, let
\[
  \Boettcher_{g_c}(z) = \Boettcher_{\chi_n(c)} \circ \eta_c(z).
\]
Then extend it to $\overline{U'} \setminus K(g_c)$ continuously by the functional equation
\begin{equation}\label{eq:Boettcher_g_c}
  \Boettcher_{g_c}(g_c(z)) = \overline{\Boettcher_{g_c}(z)}^2.
\end{equation}
By construction, the map $\Boettcher_{g_c}:(U' \setminus K(g_c)) \to \C$ is a quasiconformal homeomorphism onto its image.

\begin{lem}
  \label{lem:Boettcher}
  Fix $m>1$. On $\tilde{D}_c^m \setminus \tilde{D}_c^{m+1}$, $\Boettcher_{g_c}$ satisfies
  \[
    |\Boettcher_{g_c}(z)-z| = O(\Lambda^{-2^{-m-1}\,n}).
  \]
\end{lem}

\begin{proof}
  Note that any $z \in \tilde{D}_c^m \setminus \tilde{D}_c^{m+1}$
  satisfies $|z| \ge |\alpha_n|r_{m+1} \asymp \Lambda^{2^{-m-1}n}$
  by Lemma \ref{lem:round disks},
  and that the functional equation
  (\ref{eq:Boettcher_g_c}) implies
  $\Boettcher_{c}(z)=z+O(1/z)$ near $\infty$
  for any $c \in \C$.
  Hence by Lemma \ref{lem:qc estimate}, we have
  \begin{align*}
    |\Boettcher_{g_c}(z) -z|
     & \le |\Boettcher_{\chi_n(c)}(\eta_c(z))-\eta_c(z)|
    +|\eta_c(z)-z|                                       \\
     & =O(1/\eta_c(z))+|\eta_c(z)-z|                     \\
     & =O(1/z)+|\eta_c(z)-z|                             \\
     & =O(\Lambda^{-2^{-m-1} \,n})
    + O(\Lambda^{-(1-2^{-m+1})\,n})
    =O(\Lambda^{-2^{-m-1} \,n})
  \end{align*}
  since $m >1$.
\end{proof}

\section{Approximating filled Julia sets}
\label{sec:approx J}

Note that our parameter $c$ is close to $-2$, hence the filled Julia set of $f_c$ is close to the interval $K(f_{-2})=[-2,2]$.
In this section, we show that for $c$ close to $c_n$, $K(f_c)$ is exponentially close to $[-2,2]$ as $n \to \infty$.
In particular, in $D_c^{0} \setminus D_c^1$, the filled Julia set should almost look like two intervals on the real line (Corollary~\ref{cor:close to intervals}. See Figure~\ref{fig:per5}).
The decorations attached to the filled Julia set of the renormalization can be obtained by pulling back these ``intervals'' (Lemma~\ref{lem:argument}).

The following theorem is proved by Rivera-Letelier \cite[Theorem~C]{MR1880905} for holomorphic quadratic polynomials:
\begin{thm}[Rivera-Letelier]
  \label{thm:RL}
  Assume $Q_{c_0}$ is semihyperbolic for $c_0 \in \partial \cM$.
  Then there exists $C>0$ such that if $c \in \C$ is close to $c_0$,
  \[
    d_H(J(Q_{c_0}), K(Q_c)) \le C|c-c_0|^{1/2},
  \]
  where $d_H$ is the Hausdorff distance.
\end{thm}
Indeed, we can apply this theorem for $f_c(z) = \bar{z}^2+c$ with $c_0 = \hat{c} = -2$.
Namely. there exists $C>0$ such that
\[
  K(f_c) \subset \{|\im z| < C|c-\hat{c}|^{1/2}\}.
\]
\begin{proof}[Outline of proof of Theorem~\ref{thm:RL}]
  Since $f_{\hat{c}}^2 = Q_{\hat{c}}^2$,
  the dynamical properties in \cite[\S2, 3]{MR1880905} hold for $f_{\hat{c}}$.

  Then we can apply the proof of \cite[Proposition~4.2]{MR1880905} to
  the biquadratic family $(z^2+a)^2+b$.
  Note that we need the $\lambda$-lemma \cite{MR732343} in order to extend a holomorphic motion only to the closure of the domain of definition,
  hence we can apply the proof to a two-dimensional parameter space.
\end{proof}

Let $c=c_n+\rho_n(t)$ with $t=O(1)$ as before.
Recall that $\KoenigsInv_c = \Koenigs_c^{-1}$ can be extended to an entire function, normalized so that $\Koenigs_c(0)=1$.
Let $\tilde{K}_c := \KoenigsInv_c^{-1}(K(f_c))$.
Recall that for $c=\hat{c}=-2$.
In this section, we write $-2$ rather than $\hat{c}$ because the argument here is specific for the case $\hat{c}=-2$ and does not work for a Misiurewicz parameter $\hat{c}$ in general.
Then we have
\[
  \KoenigsInv_{-2}(w) = 2\cos\left(\frac{\pi}{2}\sqrt{w}\right).
\]
Hence it follows that $\tilde{K}_{-2} = [0,\infty]$, $\KoenigsInv_{-2}([0,4]) = [-2,2] = K_{-2}$ and,
since $D_{-2}^0 \cap K_{-2} = [-1,1]$ , it follows that
\[
  \Koenigs_{-2}(D_{-2}^0 \cap K(f_{-2}))
  = \left[\frac{4}{9}, \frac{16}{9}\right].
\]
By Theorem~\ref{thm:RL}, $\Koenigs_c(D_c^0 \cap K(f_c))$
is a compact set contained in a $O(\Lambda^{-\frac{1}{2}n})$-neighborhood
of $\tilde{K}_{-2}^0 = [\frac{4}{9},\frac{16}{9}]$.
Thus we can take compact sets $\tilde{L}_c \subset \tilde{K}_c^0 \subset \tilde{K}_c$ contained in a $O(\Lambda^{-\frac{1}{2}n})$-neighborhood of $[0,4]$ such that
\begin{align*}
  \KoenigsInv_c(\tilde{K}_c^0) & = K(f_c), &
  \tilde{L}_c = \tilde{K}_c^0 \setminus \Koenigs_c(\Int D_c^0).
\end{align*}
(See Figures~\ref{fig:per5} and \ref{fig:per5-2}.)
Then we have the following:

\begin{figure}
  \centering
  \includegraphics[width=10cm]{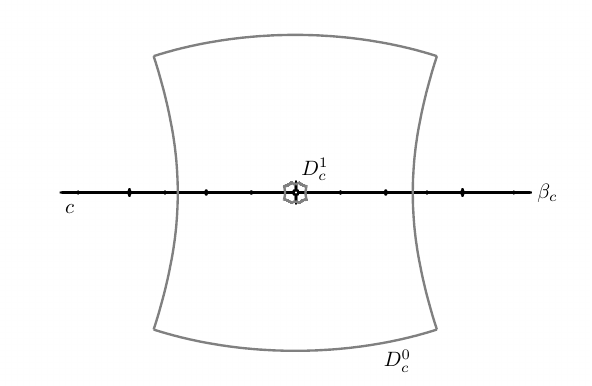}
  \caption{The filled Julia set of $f_c$, $D_c^0$ and $D_c^1$.}
  \label{fig:per5}
\end{figure}

\begin{figure}
  \centering
  \raisebox{-2cm}{\includegraphics[width=5cm]{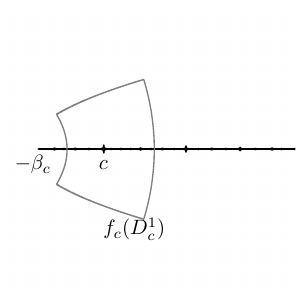}}
  $\xrightarrow{\lambda_c^{n+1}\Koenigs_c(-z)}$
  \raisebox{-2cm}{\includegraphics[width=5cm]{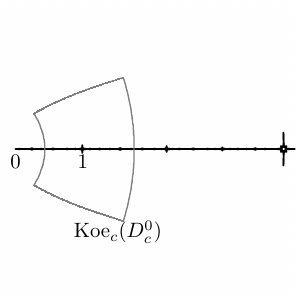}}
  \caption{Near the critical value and the linearizing coordinate.}
  \label{fig:per5-2}
\end{figure}

\begin{lem}
  For $z \in \tilde{L}_c$,
  \[
    \arg(1-z) = l\pi + O(\Lambda^{-\frac{1}{2}n})
  \]
  for some $l \in \Z$.
\end{lem}

\begin{proof}
  This simply follows from the fact that $\tilde{L}_c$ is contained in a $O(\Lambda^{-\frac{1}{2}n})$-neighborhood of $[0,4] \setminus (\frac{4}{9},\frac{16}{9})$.
\end{proof}

Now we estimate the argument of points outside $D_c^1$.
Let $L_c := K(f_c) \setminus f_c(D_c^1)$.
\begin{lem}
  If $z \in L_c$, then
  \[
    \arg(c-z) = l\pi + O(\Lambda^{-\frac{1}{2}n})
  \]
  for some $l \in \Z$.
\end{lem}

\begin{proof}
  First observe that
  \[
    \Int f_c(D_c^1) = -\KoenigsInv_c(\lambda_c^{-n-1}\Koenigs_c(\Int D_c^0)).
  \]
  Since $c$ is $O(\Lambda^{-2n})$-close to
  $c_n = -\KoenigsInv_{c_n}(\lambda_{c_n}^{-n-1})$, it is enough to show that
  for any $z \in \tilde{K}_c^0 \setminus \lambda_c^{-n-1}\Koenigs_c(\Int D_c^0)$, we have
  \[
    \arg(z-\lambda_c^{-n-1}) = l\pi + O(\Lambda^{-\frac{1}{2}n})
  \]
  for some $l \in \Z$.

  Observe that $\tilde{L}_c$ is contained in a $O(\Lambda^{-\frac{1}{2}n})$-component of $[0,\frac{4}{9}]\cup [\frac{16}{9},4]$.
  Let $\tilde{L}_c'$ be the intersection of $\tilde{L}_c$ with a small neighborhood of $[0,\frac{4}{9}]$ and $\tilde{L}_c''$ be the other part; i.e.,the intersection of $\tilde{L}_c$ with a small neighborhood of $[\frac{16}{9},4]$ (if $c \in \mathcal{M}^*$, they are simply the connected components).
  Then $z \in \tilde{K}_c^0 \setminus \lambda_c^{-n-1}\Koenigs_c(\Int D_c^0)$ satisfies either
  \begin{enumerate}
    \item $z \in \lambda_c^{-n-1}(\tilde{L}_c')$,
    \item $z \in \lambda_c^{-k}(\tilde{L}_c'')$ but not in $\lambda_c^{-k-1}(\tilde{L}_c'')$ for some $0 \le k \le n+1$.
  \end{enumerate}
  For the first case, the lemma follows immediately by the previous lemma and
  $\lambda_c^n = 16^n + O(n\Lambda^{-n})$.
  So let us consider the second case. Similarly by the previous lemma,
  we have
  \[
    \arg (z-\lambda_c^{-k}) = l\pi + O(\Lambda^{-\frac{1}{2}n}).
  \]
  Moreover, since $\tilde{L}_c''$ is $O(\Lambda^{-\frac{1}{2}n})$-close to $[\frac{16}{9},4]$, $z$ is also close to $16^{-k}[\frac{16}{9},4]$ (in particular we may take $l=0$). Hence the lemma follows because $\lambda_c^{-k}$ is close to $16^{-k}$ and $\lambda_c^{-n-1}$ is close to $16^{-n-1}<16^{-k}$ for large $n$.
\end{proof}

By taking inverse images by $f_c(z) = \bar{z}^2+c$, we have the following:
\begin{cor}
  \label{cor:close to intervals}
  for $z \in (D_c^0 \setminus D_c^1) \cap K(f_c)$,
  \[
    \arg z = l\frac{\pi}{2} + O(\Lambda^{-\frac{1}{2}n})
  \]
  for some $l \in \Z$.
\end{cor}
In fact the corollary holds for $z \in K(f_c) \setminus D_c^1$.

By applying Lemma~\ref{lem:rough approximation} repeatedly, we have the following:
\begin{lem}
  \label{lem:argument}
  Fix $m>0$. Then for $z \in (D_c^m \setminus D_c^{m+1}) \cap K(f_c)$,
  \[
    \arg z = l\frac{\pi}{2^m} + O(\Lambda^{-\frac{1}{2}n}).
  \]
\end{lem}

Namely, the filled Julia set in $D_c^m \setminus D_c^{m+1}$ is exponentially close to the straight segments of angle $l\frac{\pi}{2^m}$ in this region.
The decorations attached to the Julia set of the renormalization (see Figure~\ref{fig:decorations}) are obtained by pulling back this part by the dynamics, hence
they are close to dyadic rays.

\section{Estimate of Fatou coordinates}
\label{sec:Fatou estimate}

Let $c = c_n+\rho_n(t)$ be such that $\chi_n(c) = \omega/4$, i.e.,
$g_c$ is hybrid equivalent to $f_{\omega/4}(z) = \bar{z}^2+\frac{\omega}{4}$,
where $\omega=\frac{-1+\sqrt{3}i}{2}$ is a cubic root of unity.
The existence of such $c$ is guaranteed by the compactness of the renormalizable parameters \cite[Corollary~7.1]{MR4302166} (see also \cite[Theorem~D]{MR2970463}), the surjectivity onto hyperbolic maps \cite[Theorem~5.9]{MR4302166} (see also \cite[Theorem~C]{MR2970463}) and a standard limiting argument (see \cite[Lemma~9.2, Lemma~9.5]{MR2970463}). See also \cite{MR4357463}.

In the rest of the paper, we prove that $f_c$ satisfies the assumption of Theorem~\ref{thm:accessibility} for sufficiently large $n$, hence $c$ is accessible.
To do this end, we show the exponential convergence of the Fatou coordinate for $f_{c_n}$ to that of $f_{\omega/4}$ in this section.

We describe Fatou coordinates for $g_c$ in terms of those of $f_{\omega/4}$.
Let $\eta_c:\C \to \C$ be the quasiconformal map constructed in Section~\ref{sec:straightening}.
Recall that $\eta_c$ is a hybrid conjugacy from $F_c:U_1' \to U$ to $f_{\omega/4}$.
Since a hybrid conjugacy is holomorphic in a parabolic basin, the attracting Fatou coordinate satisfies
\[
  \Fatou_{c,\attr} = \Fatou_{\omega/4,\attr} \circ \eta_c.
\]
Therefore, the attracting Ecalle height is invariant by taking hybrid conjugacy.
In particular, the critical Ecalle height satisfies
\[
  E_{c,\attr}(c) = E_{\omega/4,\attr}(\omega/4) = 0
\]
by Lemma~\ref{lem:c=1/4}.

Let
\[
  \Fatou_{\omega/4,\rep}:V_0 (:= V_{c,\rep}) \to \C
\]
be a repelling Fatou coordinate for $f_{\omega/4}$.
Since the critical Ecalle height is zero, we may assume $\Fatou_{\omega/4,\rep}(\omega/4) = 0$.
We describe a repelling Fatou coordinate for $g_c$ in terms of $\Fatou_{\omega/4,\rep}$.
Take $n$ sufficiently large so that $\eta_c^{-1}(V_0)$ is contained in $U_1'$ and
$\Fatou_{\omega/4,\rep} \circ \eta_c$ is a quasiconformal conjugacy between $g_c$ and $z \mapsto \bar{z}+\frac{1}{2}$.

By pushing forward the standard complex structure $\sigma_0$ by $\Fatou_{\omega/4,\rep} \circ \eta_c$, we can define an almost complex structure $\sigma$ on $\C/\Z$.
Since both of the ends correspond to the filled Julia set,
$\sigma=\sigma_0$ on neighborhoods of the ends.
More precisely, the complex dilatation of $\sigma$ is supported
on a compact set independent of $n$ by Lemma~\ref{lem:convergence 2}.

By the measurable Riemann mapping theorem, there exists a quasiconformal homeomorphism
$\zeta:\C/\Z \to \C/\Z$ fixing $0$ such that $\zeta^*\sigma_0 = \sigma$.
Let $\tilde{\zeta}:\C \to \C$ be the lift of $\zeta$ fixing $0$.
Then $\Fatou_{g_c} := \tilde{\zeta} \circ \Fatou_{\omega/4,\rep}\circ \eta_c$ is a repelling Fatou coordinate for $g_c$ sending the critical value to $0$.
\begin{lem}
  \label{lem:Fatou convergence}
  For $z \in \R$, the repelling Fatou coordinate
  $\Fatou_{g_c}$ converges to $\Fatou_{\omega/4,\rep}(z)$ uniformly and exponentially as $n$ tends to $\infty$.
\end{lem}

\begin{proof}
  By Lemma~\ref{lem:qc estimate}, it suffices to show that
  $\zeta(z)-z$ converges to $0$ exponentially for $z \in \R/\Z$ as $n \to \infty$.

  Let $\xi:\C^\ast \to \C^\ast$ be the quasiconformal mapping
  satisfying $\xi(\exp(2 \pi i z))=\exp(2 \pi i \tilde{\zeta} (z))$
  and $\xi(1)=1$.
  The support of the Beltrami coefficient $\mu_\xi$ of $\xi$
  is contained in a compact annulus
  $\{z \in \C \mid 1/r^\ast \le |z| \le r^\ast\}$
  for some $r^\ast>1$ that is independent of $c=c_n+\rho_n(t)$ with $t=O(1)$.
  Moreover, by definition,
  the maximal dilatations of $\zeta$ and $\xi$
  are bounded by that of $\eta_c$,
  which is $O(\ell/R)=O(\Lambda^{-(1-2^{-m})n})$ by \eqref{eqn:dilatation of eta}.
  Hence a similar argument to the proof of Lemma~\ref{lem:qc estimate}
  yields
  \[
    \|\mu_\xi\|_p =O(\|\mu_\xi\|_\infty)=O(\ell/R)
  \]
  for some $p>2$ that is independent of both $c$ and sufficiently large $n$,
  and thus the normal solution $\tilde{\xi}$ of the Beltrami equation for $\mu_\xi$ (i.e., $\tilde{\xi}_{\bar{z}}=\mu_\xi \tilde{\xi}_z$ a.e.,
  $\tilde{\xi}(0)=0$, and $\tilde{\xi}$ is tangent to the identity at $\infty$) satisfies
  \[
    |\tilde{\xi}(z)-z|
    \le K_p \,\|\mu_\xi\|_p\cdot |z|^{1-2/p}
    = O(\ell/R)
  \]
  when $|z| \le r^\ast$.
  Since $\xi(z)=\tilde{\xi}(z)/\tilde{\xi}(1)=z+O(\ell/R)$
  for $|z|=1$, we obtain
  \begin{align*}
    |\zeta(z)-z|
     & \asymp |e^{2 \pi i \tilde{\zeta}(z)} - e^{2 \pi i z}| \\
     & =| \xi(e^{2 \pi i z}) -e^{2 \pi i z}|                 \\
     & = O(\ell/R) = O(\Lambda^{-(1-2^{-m})n})
  \end{align*}
  for $z \in \R/\Z$.
\end{proof}

\section{Proof of Theorem~\ref{thm:main}}
\label{sec:proof}

We use the notations in Section~\ref{sec:Fatou estimate}.
Recall that $\FatouInv_c=\Fatou_{c,\rep}^{-1}$ can be extended holomorphically on $\C$.
Since the critical Ecalle height for $f_c$ is zero,
the hypothesis of Theorem~\ref{thm:accessibility} for
$c_0 = c = c_n +\rho_n(t)$ is equivalent that $\FatouInv_c^{-1}(K(f_c))$ does not intersect $\R$ for sufficiently large $n$.

Let $\hat{\Psi}:\C \to \C$ be the holomorphic extension of $\FatouInv_{\omega/4}$.
Recall that
\[
  R_{\omega/4}(1/3)= \omega(\tfrac{1}{2},\infty) = \hat{\Psi}(\R)
\]
is a half line by Lemma~\ref{lem:c=1/4} and the relation $\omega f_{1/4}(\omega^{-1}z) = f_{\omega/4}(z)$.

Take $x_0 \in R_{\omega/4}(1/3)$ and consider a closed subarc $\gamma_0 \subset R_{\omega/4}(1/3)$ between $x_0$ and $f_{\omega/4}(x_0)$.
Let $r_0 := |\Boettcher_{\omega/4}(x_0)|$ and $s_0 = \Fatou_{\omega/4,\rep} \in \R$.

\begin{lem}
  \label{lem:W}
  Both $\FatouInv_c([s_0,s_0+\frac{1}{2}])$ and $\Boettcher_{g_c}^{-1}(\omega[r_0,r_0^2])$ are contained in a $O(\Lambda^{-\mu n})$-neighborhood of $\gamma_0$ for some $\mu>0$.
\end{lem}

\begin{proof}
  This is an immediate consequence of Lemma~\ref{lem:Boettcher} and Lemma~\ref{lem:Fatou convergence}.
\end{proof}

Therefore to prove that the hyperbolic component containing $c_n$ in its boundary is accessible,
it is enough to show that $\hat{K}_c := \alpha_n K(f_c)$ does not intersect this neighborhood, say $W=W_n$.

So far we have seen that convergence or growth rates are exponential.
However, the escape rate for quadratic-like maps grows super-exponentially, and this plays a critical role as follows:

Fix $d>0$ independent of $n$ and let $j_n>0$ satisfy $g_c^{j_n}(\gamma_0) \subset \tilde{D}_c^m \setminus \tilde{D}_c^{m+d}$.
By invariance, it suffices to show that $g_c^{j_n}(W)$ does not intersect $\hat{K}_c$.

Although $g_c^{j_n}$ expands $W$ exponentially, the following lemma shows, in terms of argument (equivalently, in terms of external angle), the expansion is only linearly.

\begin{lem}
  $2^{j_n} = O(n)$.
\end{lem}

\begin{proof}
  Recall that if $z \in \tilde{D}_c^m \setminus \tilde{D}_c^{m+d}$, then
  \[
    \Lambda^{\frac{1}{2^{m+d}}n} \lesssim |z| \lesssim \Lambda^{\frac{1}{2^m}n}
  \]
  for large $n$ by Lemma~\ref{lem:rough approximation}, where $A \lesssim B$ stands for $A=O(B)$.
  Hence $\Boettcher_{g_c}(z)$ also satisfies this asymptotic inequality by Lemma~\ref{lem:Boettcher}, and this implies that $n \asymp \log|\Boettcher_{g_c}(z)|$ for sufficiently large $n$.

  On the other hand, we have
  \[
    \log|\Boettcher_{g_c}(g_c^{j_n}(x_0))|= 2^{j_n}\log|\Boettcher_{g_c}(x_0)|.
  \]
  Since $x_0$, $m$ and $d$ are fixed, we have $2^{j_n} = O(n)$.
\end{proof}

Hence, by Theorem~\ref{thm:scaling limit}, we have the following:
\begin{cor}
  \label{cor:argument bound}
  For any $z \in g_c^{j_n}(W)$, $\arg z$ lies in an $O(n\Lambda^{-\mu n})$-neighborhood of
  $\frac23 \pi$.
\end{cor}

\begin{proof}
  For $z \in W$, let $\zeta = \Boettcher_{g_c}(z)$.
  Then by the definition of $W$ and Lemma~\ref{lem:W},
  $\zeta$ lies in a $O(\Lambda^{-\mu n})$-neighborhood of $\omega[r_0,r_0^2]$.

  Therefore, $\arg \zeta$ lies in a $O(\Lambda^{-\mu n})$-neighborhood of $\frac23 \pi$.
  Since
  \begin{align*}
    g_c^{j_n}(z) & = \Boettcher_{g_c}^{-1}
    \left(\zeta^{2^{j_n}}\right),
  \end{align*}
  $\arg g_c^{j_n}(z)$ lies in a $O(2^{j_n}\Lambda^{-\mu n}) = O(n\Lambda^{-\mu n})$ neighborhood of $\frac23 \pi$.
\end{proof}

Therefore by Lemma~\ref{lem:argument},
it follows that $g_c^{j_n}(W)$ does not intersect $\hat{K}_c$ for sufficiently large $n$.
Hence neither does $W$ and we have proved Theorem~\ref{thm:main}.

\section*{Declarations}

\noindent
\textbf{Conflict of interest}
On behalf of all authors, the corresponding author states that there is no conflict of interest.

\bibliographystyle{alpha}
\bibliography{accessibility}

\end{document}